\documentclass[12pt]{amsart}
\usepackage{amssymb}
\usepackage{amsmath}
\usepackage{amsfonts,latexsym,amstext}
\usepackage{xcolor}
\usepackage [T1]{fontenc}
\usepackage{graphicx}
\usepackage{amscd}
\usepackage{psfrag}
\usepackage{verbatim}
\usepackage{comment}
\usepackage{relsize}   
\usepackage[left=3cm,top=3.8cm,right=3cm]{geometry}

\usepackage[pagebackref]{hyperref}
\hypersetup{
    colorlinks,
    linkcolor={red!60!black},
    citecolor={blue!60!black},
    urlcolor={blue!90!black}
}

\newtheorem{theorem}{Theorem}[section]
\newtheorem{definition}[theorem]{Definition}

\newtheorem{lemma}[theorem]{Lemma}
\newtheorem{proposition}[theorem]{Proposition}
\newtheorem{corollary}[theorem]{Corollary}
\newtheorem{remark}[theorem]{Remark}
\newtheorem{assump}[theorem]{Assumptions}

\newcommand{\R}{{\mathbb R}}
\newcommand{\C}{{\mathcal C}}
\newcommand{\D}{{\mathcal D}}
\newcommand{\Z}{{\mathbb{Z}}}

\newcommand{\N}{{\mathbb N}}

\newcommand{\F}{{\mathcal F}}

\newcommand{\E}{{\mathbb E}}

\renewcommand{\P}{{\mathcal P}}

\newcommand{\B}{{\mathcal B}}

\newcommand{\eps}{\varepsilon}
\newcommand{\dist}{\text{dist}}
\def\wtil{\widetilde}
\def\Bk{{\mathcal B}}

\def\R{\mathbb{R}}
\def\N{\mathbb{N}}
\def\E{\mathbb{E}}

\def\F{\mathcal F}

\newcommand{\oeta}{\overline{\eta}}

\newcommand{\X}{\mathbb{X}}
\newcommand{\om}{\omega}
\newcommand{\pom}{{(\omega)}}
\newcommand{\PP}{\mathbb{P}}
\newcommand{\EE}{\mathbb{E}}

\def\sig{\sigma}

\def\lam{\lambda}

\def\Wk{{\mathcal W}}

\def\be{\begin{equation}}
\def\ee{\end{equation}}

\newcommand{\Conv}{\mathlarger{\mathlarger{\ast}}}

\title{Absolute continuity of non-homogeneous self-similar measures}

\author{Santiago Saglietti}
\address[S. Saglietti]{Faculty of Industrial Engineering and Management\\ Technion\\ Haifa, 3200003, Israel}
\email{saglietti.s@technion.ac.il}
\thanks{S.S. was partially supported by Iniciativa Cient\'\i fica Milenio NC120062 and the Israel Science Foundation research grant 1723/14}

\author{Pablo Shmerkin}
\address[P. Shmerkin]{Department of Mathematics and Statistics\\
Torcuato di Tella University and CONICET\\
Av. Figueroa Alcorta 7350 (1425), Buenos Aires\\
Argentina}
\email{pshmerkin@utdt.edu}
\thanks{P.S. was partially supported by Project PIP 11220150100355 (CONICET). Part of this work was completed while P.S. was visiting Mittag-Leffler institute as part of the program ``Fractal Geometry and Dynamics''. P.S. thanks the organizers and staff for financial support and for a stimulating atmosphere.}

\author{Boris Solomyak}
\address[B. Solomyak]{Department of Mathematics\\Bar-Ilan University\\Ramat-Gan, 5290002, Israel}
\email{bsolom3@gmail.com}
\thanks{B.S. was partially supported by the Israel Science Foundation grant 396/15}


\begin{document}

\begin{abstract}
We prove that self-similar measures on the real line are absolutely continuous for almost all parameters in the super-critical region, in particular confirming a conjecture of S-M. Ngai and Y. Wang. While recently there has been much progress in understanding absolute continuity for homogeneous self-similar measures, this is the first improvement over the classical transversality method in the general (non-homogeneous) case. In the course of the proof, we establish new results on the dimension and Fourier decay of a class of random self-similar measures.
\end{abstract}

\maketitle
	
\section{Introduction and main results}

\subsection{Non-homogeneous self-similar measures}

Given $\lam_1,\lam_2\in (0,1)$ and $p\in (0,1)$, one can form the self-similar measure $\nu_{\lam_1,\lam_2}^p$ constructed with contraction ratios $\lam_1,\lam_2$ and weight $p$. That is, $\nu_{\lam_1,\lam_2}^p$ is the only Borel probability measure $\nu$ such that
\[
\nu = p \,g_{1,\lam_1}\nu + (1-p) \,g_{2,\lam_2} \nu,
\]
where $g_{1,\lam}(x)=\lam x$ and $g_{2,\lam}(x)=\lam x+1$. Here, and throughout the article, $g\nu$ denotes the image measure: $g\nu(A)=\nu(g^{-1}A)$ for all Borel sets $A$. When $\lam_1=\lam_2$, the measure $\nu_\lam^p:=\nu_{\lam,\lam}^p$ is the classical (biased)  Bernoulli convolution.

Write
\[
s_{\lam_1,\lam_2}^p = \frac{p\log p+(1-p)\log(1-p)}{p\log\lam_1+(1-p)\log\lam_2}
\]
for the similarity dimension of $\nu_{\lam_1,\lam_2}^p$. It is well known that if $s_{\lam_1,\lam_2}^p<1$, then $\nu_{\lam_1,\lam_2}^p$ is always singular, even though if $\lam_1+\lam_2>1$ its support is an interval.

In the case of Bernoulli convolutions, it was recently proved by the second author \cite{Shmerkin16}  that there exists an exceptional set $E\subset (1/2,1)$ of zero Hausdorff dimension, such that $\nu_\lam^p$ is absolutely continuous whenever $\lam\in (1/2,1)\setminus E$ and $s_{\lam}^p:=s_{\lam,\lam}^p>1$ and, moreover, $\nu_\lam^p$ has a density in $L^q$ if $\lam^{q-1}> p^q+(1-p)^q$ (this range is sharp up to the endpoint). See Theorem 1.3, Theorem 9.1 and discussion afterwards in \cite{Shmerkin16} for details. This improved several earlier results  \cite{Solomyak95,PeresSolomyak96, Hochman14, Shmerkin14, ShmerkinSolomyak16} providing various bounds on the size of  exceptional sets for different notions of smoothness of $\nu_\lam^p$.  We remark, in particular, that in an early application of the so-called ``transversality method'', the third author already in 1995 established absolute continuity of $\nu_\lam^{1/2}$ for Lebesgue almost all $\lambda\in (1/2,1)$ \cite{Solomyak95}.  Also very recently, P. Varj\'{u} \cite{Varju16} established absolute continuity of $\nu_\lam^p$ for many explicit algebraic values of $\lambda$ near $1$.

In the general case $\lam_1\neq\lam_2$, much less is known. For $p=1/2$, J\"org Neunh\"auserer \cite{Neun_2001} and Sze-Man Ngai and Yang Wang \cite{NgaiWang05} proved that $\nu_{\lam_1,\lam_2}^{1/2}$ is absolutely continuous for almost all $\lam_1,\lam_2$ in a certain simply connected region which is very far from covering the whole super-critical parameter region $\lam_1\lam_2>1/4$ (which corresponds to $s_{\lam_1,\lam_2}^{1/2}>1$), and in particular is disjoint from a neighborhood of $(1,1)$. Ngai and Wang conjectured that, in fact, $\nu_{\lam_1,\lam_2}^{1/2}$ is absolutely continuous for almost all $\lam_1,\lam_2\in (0,1)$ such that $\lam_1\lam_2>1/4$.  This fits into the more general folklore conjecture that self-similar measures should be generically absolutely continuous in the super-critical regime. Recently, Hochman \cite{Hochman15} proved that $\nu_{\lam_1,\lam_2}^p$ has Hausdorff dimension $1$ for all $\lam_1,\lam_2$ such that $s_{\lam_1,\lam_2}^p>1$, outside of an exceptional set of Hausdorff dimension $1$, which is uniform in $p$. However, it does not seem possible to obtain absolute continuity from Hochman's approach.

The reason why the non-homogeneous case $\lam_1\neq\lam_2$ is much more difficult than the homogeneous one is that while $\nu_\lam^p$ is an infinite convolution of Bernoulli measures, $\nu_{\lam_1,\lam_2}^p$ is not. The convolution structure is crucial to all the works \cite{Solomyak95, PeresSolomyak96,Shmerkin14,ShmerkinSolomyak16, Shmerkin16}.

In this article we give a positive answer to the conjecture of Ngai and Wang, and in fact prove a stronger statement for any number of maps and arbitrary weights $\mathbf{p}$. Given $\lam=(\lam_1,\ldots,\lam_k)\in (0,1)^k$, translations $t=(t_1,\ldots,t_k)\in\R^k$, and a probability vector $\mathbf{p}=(p_1,\ldots,p_k)$, let $\nu=\nu_{\lam,t}^\mathbf{p}$ be the associated self-similar measure, i.e.
\begin{equation} \label{eq:simnu}
\nu = \sum_{i=1}^k p_i\, g_i\nu, \quad\text{where }g_i(x)=\lam_i x+t_i.
\end{equation}
Write
\[
s(\lam,\mathbf{p}) = \frac{\sum_{i=1}^k p_i\log(p_i)}{\sum_{i=1}^k p_i\log(\lam_i)}
\]
for the similarity dimension of $\nu_{\lam,t}^\mathbf{p}$, and recall that $\nu_{\lam,t}^{\mathbf{p}}$ is always singular if $s(\lam,\mathbf{p})<1$. Our main result is then the following.
\begin{theorem} \label{thm:abs-cont-non-hom}
Fix $k\ge 2$, distinct translations $t_1,\ldots,t_k\in\R$ and a probability vector $\mathbf{p}=(p_1,\ldots,p_k)$. There exists a set
\[
E_{\mathbf{p}} \subset \mathcal{R}_\mathbf{p} := \{ (\lam_1,\ldots,\lam_k)\in (0,1)^k: s(\lam,\mathbf{p}) >1\}
\]
of zero Lebesgue measure such that the following holds: for any $\lam\in \mathcal{R}_{\mathbf{p}}\setminus E_{\mathbf{p}}$, the self-similar measure $\nu_{\lam,t}^\mathbf{p}$ is absolutely continuous.
\end{theorem}
In fact, we establish this result by proving that it holds for almost all $\lambda$ in  each curve $(\lam,\lam^{\beta_2},\ldots,\lam^{\beta_k})$, where the numbers $\beta_i>0$ are fixed; see Proposition \ref{prop:non-hom}.  Unfortunately, we do not obtain any estimates on the dimension of the set of exceptional parameters, nor any information about the density. Nevertheless, this is the first result for absolute continuity of non-homogeneous self-similar reasons extending beyond what is achievable by the transversality method of \cite{Solomyak95,PeresSolomyak96,Neun_2001, NgaiWang05}.

\subsection{Outline of proof}
\label{subsec:outline-of-proof}

We give a rough outline of the proof of Theorem \ref{thm:abs-cont-non-hom} for $k=2$. In this case the translations do not play any role, and we may assume that $t_1=0$, $t_2=1$. As mentioned before, $\nu_{\lam_1,\lam_2}^p$ is not a convolution. Indeed, if one wishes to choose a point $x \in \R$ at random according to $\nu^p_{\lambda_1,\lambda_2}$, then by the self-similarity relation in \eqref{eq:simnu}, it suffices to consider an i.i.d. sequence $\tilde{\om}:=(\tilde{\om}_n)_{n \in \N} \subseteq \{1,2\}$ of Bernoulli random variables with parameter $1-p$ and take $x$ as the unique element in $\R$ satisfying
\begin{equation}
\label{eq:xnu}
\{x\}= \bigcap_{n \in \N} B_{\tilde{\om}|n}
\end{equation} where $\tilde{\om}|n =(\tilde{\om}_1,\dots,\tilde{\om}_n) \in \{1,2\}^n$ and, for each $u \in \{1,2\}^n$, we put
\[
B_u:=g_{u_1} \circ \dots \circ g_{u_n} ([-R,R])
\]
for $R > 0$ sufficiently large so as to guarantee that $g_i([-R,R])\subset [-R,R]$ for $i=1,2$ (the fact that $\lambda_1,\lambda_2 < 1$ allows us to find such an $R$). From this description it follows that $\nu_{\lambda_1,\lambda_2}^p$ is the distribution of the random sum
\begin{equation}\label{eq:nusum}
\sum_{n \in \N} \left( \prod_{j=1}^{n-1} \lambda_{\tilde{\om}_j}\right) t_{\tilde{\om}_n}
\end{equation}
with the convention $\prod_{j=1}^0 \lambda_{\tilde{\om}_j} = 1$, from which one sees that $\nu_{\lambda_1,\lambda_2}^p$ is not a convolution since the summands in \eqref{eq:nusum} are not independent.
Nevertheless, there exists a decomposition
\begin{equation} \label{eq:nudecomp}
\nu_{\lam_1,\lam_2}^p = \int  \eta^\pom \,d\PP(\om),
\end{equation}
where the $\eta^\pom=\eta^\pom_{\lam_1,\lam_2,p}$ are random measures which are not strictly self-similar, but still possess a kind of stochastic self-similarity (see \eqref{eq:dssr} below) and, crucially, do have a convolution structure. This decomposition can be described as follows. Pick $r \in \N$ and notice that it makes no difference if in \eqref{eq:xnu} we take instead the intersection in ``blocks of $r$ bits'', i.e.
\begin{equation}
\label{eq:xnu2}
\{x\} = \bigcap_{n \in \mathbb{N}} B_{\tilde{\omega}|rn}.
\end{equation} This corresponds to thinking of $\tilde{\omega}$ as a sequence $\overline{\om}=(\overline{\om}_n)_{n \in \N}$ in $\left( \{ 1,2\}^r \right)^\mathbb{N}$, where each $\overline{\omega}_i$ is chosen independently according to the rule $\mathbf{P}(\overline{\omega}_i = u ) =  p^{(u)_1}(1-p)^{r-(u)_1}=:p_u$ for each $u \in \{1,2\}^r$, where we set $(u)_1:=| \{ i : u_i = 1 \}|$. Now, for each $k=0,\dots,r$ we define the weight
\[
q_k:= \sum_{u \in \{1,2\}^r : (u)_1=k} p_u = p^k(1-p)^{r-k}|\{u : (u)_1=k\}|.
\]
Then, from the conditional probability rule
\[
p_u=\mathbf{P}((\overline{\omega}_i)_1=(u)_1)\mathbf{P}(\overline{\omega}_i=u|(\overline{\omega}_i)_1=(u)_1) = q_{(u)_1} \mathbf{P}(\overline{\omega}_i=u|(\overline{\omega}_i)_1=(u)_1)
\]
it follows that each $\overline{\om}_i$ can be chosen in a $2$-step procedure, by first selecting the value of $\omega_i:=(\overline{\om}_i)_1$ according to the probability vector $q=(q_0,\dots,q_r)$ and then choosing the ``final'' value of $\overline{\om}_i$ uniformly among all possibilities $u$ with $(u)_1=\omega_i$. In conclusion, we may choose $x$ at random according to $\nu^p_{\lambda_1,\lambda_2}$ by following this $3$-step procedure:
\begin{enumerate}
	\item [i.] We choose first $\omega=(\omega_i)_{i \in \N} \in \{0,\dots,r\}^\N$ according to the Bernoulli product measure $\PP$ with marginals $q$.
	\item [ii.] Given $\omega$, we choose each $\overline{\omega}_i$ for each $i \in \N$ independently uniformly among all possibilities allowed by the value of $\omega_i$.
	\item [iii.] Having obtained $\overline{\omega}$, we take $x$ as in \eqref{eq:xnu2}.
\end{enumerate} If for each $\omega \in \{0,\dots,r\}^\N$ we define $\eta^\pom$ as the measure which selects a point $x$ at random according to (ii)+(iii) above, then it is clear that \eqref{eq:nudecomp} holds. To see that $\eta^{\pom}$ is indeed a convolution, we notice that by \eqref{eq:xnu2} the measure $\nu^p_{\lambda_1,\lambda_2}$ is the distribution of the random sum
\begin{equation}
\label{eq:xsum3}
\sum_{n \in \N} \left( \prod_{j=1}^{n-1} \lambda_{\overline{\om}_j}\right) g_{\overline{\om}_{n}}(0)
\end{equation} for $\lambda_{\overline{\om}_j}:= \lambda_{\tilde{\om}_{r(j-1)+1}}\dots \lambda_{\tilde{\om}_{rj}}$ and $g_{\overline{\om}_{n}}:=g_{\tilde{\om}_{r(j-1)}} \circ \dots \circ g_{\tilde{\om}_{rj}}$, and thus that $\eta^{\pom}$ is the distribution of the random sum in \eqref{eq:xsum3} given the values of the sequence $\omega$. But since the value of $\lambda_{\overline{\omega_j}}$ depends on $\overline{\omega}_j$ only through $\omega_j=(\overline{\omega}_j)_1$, it follows that $\eta^\pom$ is a convolution since the summands in \eqref{eq:xsum3} are independent once the value of $\omega$ is fixed.

We then prove absolute continuity of  $\nu_{\lam_1,\lam_2}^p$  by showing that $\eta^\pom$ is absolutely continuous $\PP$-almost surely (for a.e. $\lam_1,\lam_2$), and we do this by following the approach from \cite{Shmerkin14}. Namely, exploiting the convolution structure we decompose each $\eta^\pom =(\eta')^\pom*(\eta'')^\pom$, where $(\eta')^\pom$ and $(\eta'')^\pom$ are random self-similar measures in the same general class. More precisely, recall from \eqref{eq:xsum3} that $\eta^\pom$ is the distribution of a random sum $\sum_{n=1}^\infty X_n^\pom$ of independent discrete random variables (they also depend on $r$). We define $(\eta')^\pom$ as the distribution of $\sum_{s \mid n} X_n^\pom$, and $(\eta'')^\pom$ as the distribution of $\sum_{s \nmid n} X_n^\pom$, for a suitable $s\in\N$.

To conclude, we prove that $\PP$-almost surely $(\eta')^\pom$ has power Fourier decay and $(\eta'')^\pom$ has full Hausdorff dimension, provided $r$ and then $s$ were taken large enough. Once this is achieved, absolute continuity follows from a Lemma from \cite{Shmerkin14}.

Proving the required properties of  $(\eta')^\pom$ and $(\eta'')^\pom$ involves extending to this random setting a number of results that were known in the setting of deterministic self-similar measures: the Erd\H{o}s-Kahane method for obtaining power Fourier decay, and a theorem of Hochman \cite{Hochman14} giving mild sufficient conditions for a self-similar measure to have full dimension. While we follow the strategies of the deterministic case, obtaining these extensions involves a fair amount of work, so these will take up most of the paper.

\subsection{A class of random self-similar measures}

As explained above, our approach depends on the study of a class of random measures which we now introduce formally. They are closely related to $1$-variable fractals (also sometimes called homogeneous random fractals). The results we obtain for these measures may be of independent interest.

We will work with a finite set $I$ of iterated function systems of similarities $\Phi^{(i)}=(f_1^{(i)}, \dots, f_{k_i}^{(i)})$, $i\in I$. We assume throughout that each IFS is homogeneous and uniformly contracting, i.e. the maps $f^{(i)}_j$ are of the form
\[
f^{(i)}_j (x)= \lambda_i x + t^{(i)}_j
\]
for certain constants $\lambda_i \in (0,1)$ and $t^{(i)}_j \in \R$. We emphasize that we allow $k_i$ to be $1$ for some $i$, i.e. to have degenerate iterated function systems.  Notice that, since the maps $f_j^{(i)}$ are uniformly contractive, if $R>0$ is sufficiently large then $f_j^{(i)}([-R,R])\subset [-R,R]$ for all $i\in I$, $j\in\{1,\ldots, k_i\}$.

Given a sequence $\omega=(\omega_n)_{n \in \N} \in \Omega:=I^{\N}$, we define the space of words of length $n$ (possibly with $n=\infty$) with respect to $\omega$ by the formula
\[
\mathbb{X}^{(\omega)}_n:=  \prod_{j=1}^n \{1,\dots,k_{\omega_j}\}.
\]
Note that all $\mathbb{X}^{(\omega)}_n$ are subsets of a common tree $\X_n:=\prod_{j=1}^n \{1,\ldots,k_{\max}\}$, for $k_{\max}:=\max_{i\in I} k_i$. For each $n \in \N$ and $u \in \X_n^\pom$ we consider the interval
\begin{equation} \label{eq:def-symbolic-interval}
B^{(\omega)}_{u}:=f^{(\omega)}_u([-R,R]),
\end{equation}
where $f^{(\omega)}_u:=f^{(\omega_1)}_{u_1} \circ \dots \circ f^{(\omega_n)}_{u_n}$. Furthermore, we define the compact set
\[
\mathcal{C}^{(\omega)}:= \bigcap_{n \in \N} \bigcup_{u \in \mathbb{X}^{(\omega)}_n} B^{(\omega)}_{u}.
\]
Note that, for every $n$, we have the inclusion $B^{(\omega)}_{ul} \subset B^{(\omega)}_{u}$, for each $u \in \mathbb{X}^{(\omega)}_n$ and  $l \in \{1,\dots,k_{\omega_{n+1}} \}$ (where $ul$ denotes the concatenation of $u$ and $l$). In other words, these intervals are nested. Moreover, their diameters tend to zero as $n\to\infty$ uniformly over $u\in\mathbb{X}_n^\pom$. Alternatively, we have that  $\C^{\pom}=\Pi_{\omega}\left(\mathbb{X}^{(\omega)}_\infty\right)$, where $\Pi_\omega$ is the \emph{coding map} given by
\[
\{\Pi_\om(u)\} = \bigcap_{n=1}^\infty B^\pom_{u|n},
\]
where $u|n$ is the restriction of the infinite word $u \in \X^\pom_\infty$ to its first $n$ coordinates. Equivalently, the coding map is defined by the formula
\[
\Pi_\om(u):= \sum_{n \in \N} \left(\prod_{j=1}^{n-1}\lambda_{\om_j}\right)t^{(\om_n)}_{u_n}
\]
with the convention that $\prod_{j=1}^{0}\lambda_{\om_j} = 1$.

Given $u\in\X_n^\pom$, we also define the \emph{cylinder} $[u]_\om$ as the set of infinite words in $\X_\infty^\pom$ which start with $u$, and note that $\Pi_\om([u]_\om)\subset B_u^\pom$. We remark that we do not assume that the intervals $\{ B^{(\omega)}_{u}:u \in \mathbb{X}^{(\omega)}_n\}$ are disjoint or any other separation condition.  Also, even though $\mathcal{C}^{(\omega)}$ is defined for every $\omega \in \Omega$, \mbox{in the sequel} we will be drawing $\omega$ according to a probability measure $\PP$, which we shall refer to as the \textit{selection measure}.

We will not be interested in the sets $\mathcal{C}^{\pom}$ themselves, but rather in measures supported on them. For each $i\in I$, let $p_i=(p_1^{(i)},\ldots,p_{k_i}^{(i)})$ be a probability vector with strictly positive entries. On each $\X_\infty^\pom$ we can then define the product measure
\[
\oeta^\pom :=  \prod_{n=1}^\infty p_{\omega_n}.
\]
The projection of $\oeta^\pom$ via the coding map is a Borel probability measure $\eta^\pom$ supported on $\C^\pom$. Equivalently, we may define $\eta^\pom$ as the distribution of the random sum $\Pi_\om((U_n)_{n \in \N})$, where $(U_n)_{n \in \N}$ is a sequence of independent random variables satisfying for every $n \in \N$ and $j=1,\dots,k_{\om_n}$
\[
\mathbf{P}(U_n= j) = p^{(\om_n)}_j.
\]
In the deterministic case $|I|=1$, $\eta^\pom$ is simply a homogeneous self-similar measure on $\C^\pom$. Furthermore, notice that in general $\eta^\pom$ satisfies the following ``dynamic self-similarity'' relation
	\[
	\eta^\pom = \sum_{u \in \X^\pom_1} p_u^{(\om_1)} \cdot f^{(\om)}_{u}\eta^{(T\om)}
	\]
where $T$ denotes the left-shift on $\Omega$. Iterating this, we get
\begin{equation} \label{eq:dssr}
\eta^\pom = \sum_{u \in \X^\pom_k} p_u^{\om} \cdot f^{(\om)}_{u}\eta^{(T^k\om)}
\end{equation}
for each $k\in\N$, where $p_u^{\om}=p_{u_1}^{\om_1}\cdots p_{u_k}^{\om_k}$.

Finally, in the sequel we shall refer to the triple $\Sigma:= \left((\Phi^{(i)})_{i\in I},(p_i)_{i \in I},\PP\right)$ as the \textit{model} under consideration.

The following are our main results on the measures $\eta^\pom$.
\begin{theorem} \label{thm:exact}
Let $\Sigma$ be a model for which $\PP$ is $T$-invariant and ergodic. Then there exists $\alpha\in[0,1]$ such that for $\PP$-almost all $\om$, the measure $\eta^\pom$ has exact dimension $\alpha$, that is,
\[
\lim_{r\downarrow 0} \frac{\log(\eta^\pom (B(x,r)))}{\log r}  = \alpha
\]
for $\eta^\pom$-almost all $x$.
\end{theorem}
We will call the value of $\alpha$ given by this theorem the \emph{dimension} of the model $\Sigma$, and denote it $\dim(\Sigma)$. 

The proof of Theorem \ref{thm:exact} will reveal also that the projections $\Pi_\om$ are almost surely dimension-conserving: see Corollary \ref{cor:dimension-conservation}.

Given a model $\Sigma$, we define the \textit{similarity dimension} $\textup{s-dim}(\Sigma)$ of $\Sigma$ as
\[
\textup{s-dim}(\Sigma) := \left(\int_\Omega \log (\lambda_{\om'_1}) d\PP (\om')\right)^{-1} \int_\Omega \sum_{j=1}^{k_{\om'_1}} p_j^{(\om'_1)} \log p_j^{(\om'_1)} d\PP (\om').
\]
As we will see in Section \ref{sec:exact}, this number is a ``candidate'' for the dimension $\dim(\Sigma)$, and it is always an upper bound for $\dim(\Sigma)$. For example, if $|I|=1$, it agrees with the standard similarity dimension of a self-similar measure. One way in which one can have a strict inequality $\dim(\Sigma)<\textup{s-dim}(\Sigma)$ is if there is an exact overlap, that is, if there exist $n \in \N$, $\om\in\Omega$ and $u,v \in \X^\pom_n$ such that $f^\pom_u = f^\pom_v$ and $u \neq v$.\footnote{This is equivalent to there being exact coincidences $f^\pom_u = f^\pom_v$ for $u$ and $v$ of not necessarily the same length.} There is also a trivial strict inequality if $\textup{s-dim}(\Sigma)>1$.

For deterministic self-similar measures, Hochman \cite[Theorem 1.1]{Hochman14} proved that the dimension equals the minimum between the similarity dimension and $1$ if a \emph{quantitative} non-overlapping condition holds. We obtain an analogous result for the random measures $\eta^\pom$. For $u,v \in \X_n^\pom$ define
\[
d^\pom(u,v):= |f^\pom_u(0) - f^\pom_v(0)|
\] and
\[
\Delta_n^\pom = \Delta_n^\pom(\Sigma) :=\left\{\begin{array}{ll} \min \{ d^\pom(u,v) : u,v \in \X_n^\pom, u \neq v\} & \text{ if $|\X^\pom_n| > 1$ } \\  0 & \text{ if $|\X^\pom_n|=1.$}\end{array}\right..
\]
Note that exact overlapping occurs for $\eta^\pom$ if and only if $\Delta_n^\pom = 0$ and $|\X_n^\pom|>1$ for some $n \in \N$ (equivalently, all sufficiently large $n$). The following is the extension of \cite[Theorem 1.1]{Hochman14} to our class of random self-similar measures.

\begin{theorem}\label{thm:sup-exp-conc} If $\PP$ is a globally supported Bernoulli probability measure\footnote{In the published version $\PP$ was an arbitrary ergodic shift-invariant probability measure. However, the proof implicitly used that the measure is Bernoulli in two places. The proof of the main Theorem 1.1. is not affected by this change, since it relies on the case of $\PP$ Bernoulli. For details, see ``Absolute continuity of self-similar measures on the plane''  by B. Solomyak and A. \'Spiewak, arXiv:2301.10620, which also contains an argument allowing one to recover Theorem 1.3
for an arbitrary ergodic shift-invariant probability measure.}
on $\Omega$, then any model $\Sigma$ with selection measure $\PP$ such that $k_i\ge 2$ for some $i \in I$ satisfies the following:

If $\dim(\Sigma) < \min\{1,\textup{s-dim}(\Sigma)\}$ then
	\[
	\frac{\log \Delta^{(\cdot)}_n}{n} \overset{\PP}{\longrightarrow} -\infty,
	\] i.e. for every $M > 0$ one has
	\[
	\lim_{n \rightarrow +\infty} \PP \left( \left\{ \om \in \Omega :  \frac{\log \Delta^{(\om)}_n}{n} \leq - M \right\} \right)=1.
	\]
\end{theorem}

Our final main result concerns Fourier decay of $\eta^\pom$. Write $\P_1$ for the Borel probability measures on the real line. The Fourier transform of $\mu\in\P_1$  is
\[
\widehat{\mu}(\xi)= \int e^{i\pi x\xi}\,d\mu(x).
\]
(We choose this slightly unusual normalization for technical reasons.) Denote
\[
\D_1(\sigma) = \{\mu\in \P_1:\, |\widehat{\mu}(\xi)| = O_\mu(|\xi|^{-\sigma})\}\quad \text{and}\quad \D_1 = \bigcup_{\sig>0} \D_1(\sig).
\]
That is, $\D_1$ is the set of measures with power Fourier decay. (We note that belonging to $\D_1(\sigma)$ guarantees a minimum decay rate for all large frequencies; faster decay for some or all frequencies is allowed.)

We note that if a model is such that for each $i$ all the translations $t_j^{(i)}, j\in\{1,\ldots,k_i\}$ are equal (this includes the possibility $k_i=1$), then the generated measures are all a single atom, so we have to exclude this class of models in order to have any chance of obtaining power Fourier decay.

\begin{definition} \label{def:non-degenerate}
Let us say that a model $\Sigma=\left((\Phi^{(i)})_{i\in I},(p_i)_{i\in I},\PP\right)$ is \emph{non-degenerate} if there exist $i\in I$ and $1\le u_1<u_2\le k_i$ such that $t_{u_1}^{(i)}\neq t_{u_2}^{(i)}$ (in particular, $k_i\ge 2$).
\end{definition}

Note that this definition is independent of the $\lam_i$ and the selection measure $\PP$. The next theorem asserts that almost all measures generated by ``nearly all'' non-degenerate models with Bernoulli selection measure have power Fourier decay. The precise statement is fairly technical, and we need to keep track of the measurability of the relevant sets in order to be able to reverse the order of quantifiers at a later point.

\begin{theorem} \label{thm:Fourier-decay}
Let $I$ be a finite set, and let $(\beta_i)_{i\in I}$ be strictly positive numbers. Also, let $\PP$ be a Bernoulli measure on $\Omega=I^\N$.

Then there exists a Borel set $\mathcal{G}\subset \Omega\times (0,1)$ such that the following hold:
\begin{enumerate}
\item [(i).]
For $\PP$-almost all $\om$,
\[
\dim_H\{ \lam: (\om,\lam)\notin\mathcal{G} \} = 0.
\]
\item [(ii).]
If $(\om,\lam)\in \mathcal{G}$,  and $\Sigma=\bigl((\Phi^{(i)})_{i\in I},(p_i)_{i\in I}, \PP\bigr)$ is any non-degenerate model with selection measure $\PP$, such that the contraction ratio of the maps in $\Phi^{(i)}$ is $\lam^{\beta_i}$, then $\eta^\pom\in\mathcal{D}_1$.
\end{enumerate}

\end{theorem}

This article is organized as follows. In Section \ref{sec:notation} we fix some notational conventions. We establish Theorems \ref{thm:exact}, \ref{thm:sup-exp-conc}  and \ref{thm:Fourier-decay} in Sections \ref{sec:exact}, \ref{sec:dimension} and \ref{sec:EK}, respectively. We complete the proof of Theorem \ref{thm:abs-cont-non-hom} in Section \ref{sec:abs-cont}. Sections \ref{sec:exact}--\ref{sec:abs-cont} can be read independently of each other.

\section{Notation}
\label{sec:notation}

In this section we introduce notation to be used throughout the rest of the article. Let $\Sigma=\bigl((\Phi^{(i)})_{i\in I},(p_i)_{i\in I}, \PP\bigr)$ be a model as above, with  $\Phi^{(i)}=(f_1^{(i)}, \dots, f_{k_i}^{(i)})$, $i\in I$, and
\[
f^{(i)}_j (x)= \lambda_i x + t^{(i)}_j.
\]
We set $\lambda_{\max}=\max\{\lambda_i:i\in I\}$ and $\lambda_{\min}=\min\{\lambda_i:i\in I\}$.

We have already defined $f^{(\omega)}_u:=f^{(\omega_1)}_{u_1} \circ \dots \circ f^{(\omega_n)}_{u_n}$ and $p_u^{\om}=p_{u_1}^{\om_1}\cdots p_{u_k}^{\om_k}$ for $u\in\X_n^\pom$. We also write
\[
f^\pom_u(x) = \lam^\pom x + t_u^{\pom}
\]
for the contraction ratio and translation parts of $f^\pom_u(x)$. Note that, with this notation, we have
\be \label{eq:def-exp-concentration}
\Delta_n^\pom(\Sigma) = \min_{u\neq v\in\X_n^\pom} |t_u^\pom-t_v^\pom|,
\ee
provided $|\X_n^\pom|>1$.

We will often need to consider parametrized families of models, or several related models at once. These will be denoted by $\Sigma_{\lam}$, $\Sigma'$, etc, with corresponding adaptations for each of the components of the model.

The notation $\PP$ will always refer to the selection measure for a model, while $\EE$ will refer to expectation with respect to $\PP$. Probability and expectation for any other probability spaces will be denoted by $\mathbf{P},\mathbf{E}$, respectively.

We use Landau's $O(\cdot)$ notation: by $A=O(B)$ we mean that $|A| \le C B$ for some positive constant $C>0$. If $C$ is allowed to depend on some parameter, this is indicated as a subscript; thus, $A=O_R(B)$ means that $|A|\le C(R) B$ for some positive function $C(R)$.

The set of Borel probability measures on a set $A\subseteq \R$ will be denoted $\mathcal{P}(A)$. Finally, given $r>0$, $x \in \R$ and a measure $\nu\in\mathcal{P}(\R)$, we denote the push-forward of $\nu$ under the map $f(z)=rz+x$ by $r\cdot\nu+x$.

Other than the notation introduced up to this point, the notation in each of the later sections is independent of each other, so that the same symbol may refer to different concepts in different sections.

\section{Exact dimensionality and dimension conservation}
\label{sec:exact}

In this section we prove Theorem \ref{thm:exact}. Consider the product space
\[
\Omega \times \mathbb{X}_\infty = \left(I\times \{1,\dots,k_{\max}\}\right)^\N
\]
and let $\overline{\PP}$ be the probability measure on $\Omega \times \mathbb{X}_\infty$ given by the formula
\begin{equation} \label{eq:exact0}
\int_{\Omega \times \mathbb{X}_\infty} f(\om,u) d\overline{\PP}(\om,u) = \int_\Omega  \int_{\mathbb{X}_\infty} f(\om,u) d\overline{\eta}^\pom(u) d\PP(\om)
\end{equation} for any bounded measurable function $f: \Omega \times \mathbb{X}_\infty \to \R$.

For simplicity, we abuse the notation slightly and use $T$ to denote both the left-shift on $\Omega$ and on $\Omega \times \mathbb{X}_\infty$. Throughout this section, it will be convenient to define $p_j^{(i)}=0$ if $j>k_i$. We start by showing that $\overline{\PP}$ is $T$-invariant and ergodic if $\PP$ is.

\begin{lemma} \label{lem:ovP-ergodic}
 If $(\Omega,\mathcal{B}(\Omega),\PP,T)$ is an ergodic measure preserving system, then so is $(\Omega \times \mathbb{X}_\infty,\mathcal{B}(\Omega\times\mathbb{X}_\infty),\overline{\PP},T)$.
\end{lemma}
\begin{proof}
Given $(\om,u)\in I^n\times \{1,\ldots,k_{\max} \}^n$, let $[\om,u]$ denote the corresponding cylinder set:
\[
[\om,u] = \{ (\om',u')\in \Omega\times\mathbb{X}_\infty: \om'_{i}=\om_{i}, u'_i=u_i \text{ for } i=1,\ldots,n\}.
\]
Cylinder sets form a semi-algebra that generates $\mathcal{B}(\Omega\times\mathbb{X}_\infty)$. It follows from \eqref{eq:exact0} that
\begin{align*}
\overline{\PP}([\om,u]) &= \PP([\om]) \, p_{u_1}^{(\om_1)} \cdots p_{u_n}^{(\om_n)},\\
\overline{\PP}(T^{-1}[\om,u]) &= \PP(T^{-1}[\om]) p_{u_1}^{(\om_1)} \cdots p_{u_n}^{(\om_n)},
\end{align*}
so $\overline{\PP}$ is $T$-invariant. Likewise, if $[\om,u]$, $[\om',u']$ are cylinders of lengths $m,m'$ respectively, then for any $j> m$,
\[
\overline{\PP}([\om,u]\cap T^{-j}([\om',u'])) = \PP([\om]\cap T^{-j}[\om']) p_{u_1}^{(\om_1)} \cdots p_{u_m}^{(\om_m)}\,\cdot\, p_{u'_1}^{(\om'_1)} \cdots p_{u'_{m'}}^{(\om'_{m'})},
\]
so that, using \cite[Theorem 1.17]{Walters82},
\begin{align*}
\lim_{n\to\infty}\frac1n \sum_{j=0}^{n-1} &\overline{\PP}([\om,u]\cap T^{-j}[\om',u'])\\
 &= p_{u_1}^{(\om_1)} \cdots p_{u_m}^{(\om_m)}\,\cdot\, p_{u'_1}^{(\om'_1)} \cdots p_{u'_{m'}}^{(\om'_{m'})}\lim_{n\to\infty}\frac1n \sum_{j=0}^{n-1} \PP([\om]\cap T^{-j}([\om']))\\
  &= p_{u_1}^{(\om_1)} \cdots p_{u_m}^{(\om_m)}\,\cdot\, p_{u'_1}^{(\om'_1)} \cdots p_{u'_{m'}}^{(\om'_{m'})}\PP([\om])\PP([\om'])\\
  &= \overline{\PP}([\om,u])\overline{\PP}([\om',u']),
\end{align*}
which, relying on \cite[Theorem 1.17]{Walters82} again, shows that $\overline{\PP}$ is $T$-ergodic.
\end{proof}

We now start the proof of Theorem \ref{thm:exact}. We follow the ideas in \cite{furstenberg2008ergodic,FalconerJin14} which deal with related, but different, models of random measures. We shall begin by showing that the measures $\overline{\eta}^\pom$ are exact-dimensional.
\begin{lemma}\label{lema:exact} If $\PP$ is ergodic, then for $\overline{\PP}$-almost every pair $(\om,u) \in \Omega \times \mathbb{X}_\infty$
	\[
	\lim_{n \rightarrow +\infty} \frac{- \log( \overline{\eta}^\pom([u|n]_\om) )}{n} = \int_{\Omega \times \mathbb{X}_\infty} -\log (p^{(\om'_1)}_{u'_1}) d\overline{\PP}(\om',u')=:\alpha_1.
	\]
\end{lemma}

\begin{proof} The lemma is a straightforward consequence of the ergodic theorem for the pair $(\overline{\PP},T)$, since for any $(\om,u) \in \Omega \times \mathbb{X}_\infty$ one has $\overline{\eta}^\pom([(u_1,\dots,u_n)]_\om) = \prod_{i=1}^n p^{(\om_i)}_{u_i}$, so that
	\begin{equation} \label{eq:exact1}
	\frac{ \log( \overline{\eta}^\pom([u|n]_\om)) }{n} = \frac{1}{n} \sum_{i=1}^n \log( p^{(\om_i)}_{u_i}) = \frac{1}{n} \sum_{i=0}^{n-1} f(T^i(\om,u)),
	\end{equation} where $f: \Omega \times \mathbb{X}_\infty \to \R$ is the function given by $f(\om,u):= \log(p^{(\om_1)}_{u_1})$. We observe that, since $
	u_1\in\{1,\ldots,k_{\om_1}\}$ for $\overline{\mathbb{P}}$-a.e. $(\om,u)$, then $f$ is $\overline{\PP}$-integrable so that the ergodic theorem can truly be applied.
\end{proof}

\begin{remark}\label{rem:exact} Since $\overline{\PP}$ is given by \eqref{eq:exact0}, it follows from Lemma \ref{lema:exact} that for $\PP$-almost every $\om \in \Omega$ the measure $\overline{\eta}^\pom$ is exact-dimensional in the sense that the limit in the left-hand side of \eqref{eq:exact1} exists for $\overline{\eta}^\pom$-almost every $u$ and is independent of $u$. We say then that the exact dimension of $\overline{\eta}^\pom$ exists and equals $\textup{dim}\, \overline{\eta}^\pom :=\alpha_1$.
\end{remark}

 In the sequel, it will be convenient to consider the following joint construction of $(\overline{\eta}^\pom)_{\om \in \Omega}$. First, let $(O,\F,\mathbf{P})$ be a probability space on which one has defined an array $\{ U^{(i)}_n : i\in I\,,\,n \in \N\}$ of independent random variables, each $U^{(i)}_n$ with distribution $p_i$ for every $n \in \N$. Then, for each $\om \in \Omega$ consider the random sequence $\overline{X}^\pom=(\overline{X}^\pom_n)_{n \in \N}$ given by the formula $\overline{X}^\pom_n:= U^{(\om_n)}_n$ for all $n \in \N$. Define also for each $\om \in \Omega$ and $n \in \N_0$ the random variables $T^n X^\pom := \Pi_{T^n \om}( T^n \overline{X}^\pom)$, where we write $T$ to denote both the left-shift on $\Omega$ and on $\X_\infty$.
 For simplicity, we denote $T^0 X^\pom$ by $X^\pom$.
 Let us observe that, for each $\om \in \Omega$ and $n \in \N$, this construction has the following properties:
\begin{enumerate}
	\item [PI.] $\overline{X}^\pom$ and $X^\pom$ have distribution $\overline{\eta}^\pom$ and $\eta^\pom$, respectively.
	\item [PII.] $T^n\overline{X}^\pom$ and $T^nX^\pom$ have distribution $\overline{\eta}^{(T^n\pom)}$ and $\eta^{(T^n\pom)}$, respectively.
	\item [PIII.] The random vector $(\overline{X}^\pom_1,\dots,\overline{X}^\pom_n)$ is independent of $T^n \overline{X}^\pom$.
\end{enumerate}
Moreover, one has the following classical result on the existence of a regular conditional probability of $\overline{X}^\pom$ given $X^\pom$; see \cite{Simmons12} for a clean proof.
\begin{theorem}\label{thm:rcp} For all $\om\in\Omega$ there exists a function $Q^\pom: \R \times \B(\X^\pom_\infty)  \rightarrow [0,1]$ such that:
	\begin{enumerate}
		\item [(i).] For $\eta^\pom$-almost every $x$ the map $Q^\pom(x,\cdot)$ is a probability measure on $(\X^\pom_\infty,\B(\X^\pom_\infty))$ supported on the fiber $\Pi^{-1}_\om(\{x\})= \{u \in \X^\pom_\infty: \Pi_\om(u)=x\}$.
		\item [(ii).] For every $B \in \B(\X^\pom_\infty)$ the mapping $Q^\pom(\cdot,B)$ is Borel measurable and satisfies
		\[
		\mathbf{P}( \overline{X}^\pom \in B) = \int Q^\pom (X^\pom,B) d\mathbf{P}.
		\]
\item [(iii).] For each $B \in  \B(\X^\pom_\infty)$, one has that for $\eta^\pom$-almost every $x$
		\[
		Q^\pom(x,B) = \lim_{r \rightarrow 0} \mathbf{P}(\overline{X}^\pom \in B\,|\,X^\pom \in B(x,r)). 
		\]
	\end{enumerate}
	\end{theorem}

Let us now return to the proof of Theorem \ref{thm:exact}. We will show that in fact $\alpha$ is given by
	\begin{equation} \label{eq:value-exact-dim}
	\alpha =  \frac{\int_{\Omega \times \mathbb{X}_\infty} \left[ -\log (p^{(\om'_1)}_{u'_1}) + \log \left( \mathbf{P}( \overline{X}^{(\om')}_{1} = u'_1 | X^{(\om')}= \Pi_{\om'}(u'))\right) \right] d\overline{\PP}(\om',u')}{\int_\Omega -\log(\lambda_{\om'_1}) d\PP(\om')}.
	\end{equation}	
Here $\mathbf{P}( \overline{X}^{(\om')}_{1} = u'_1 | X^{(\om')}= \Pi_{\om'}(u'))$ is defined as in the right-hand side of Theorem \ref{thm:rcp}(iii), i.e. it equals
\[
\lim_{r\to 0} \mathbf{P}( \overline{X}^{(\om')}_{1} = u'_1 | X^{(\om')} \in B( \Pi_{\om'}(u'),r)).
\]
In other words, $\dim(\Sigma)$ equals its similarity dimension, minus a quantity (the quotient of  ``fiber entropy'' and the Lyapunov exponent) that measures, in some sense, the size of the overlaps. To show this, for each $n \in \N$ and $u \in \mathbb{X}_\infty$ define
\[
g_n(\om,u):= \left\{\begin{array}{ll}-\log \left( \mathbf{P}( \overline{X}^\pom_1 = u_1 | X^\pom \in B_n(\om,u))\right) & \text{ if $u \in \X^\pom_\infty$} \\ \\ 0 & \text{ otherwise},\end{array}\right. 
\]
where $B_n(\om,u):=B_n(\om,\Pi_\om(u))=B(\Pi_\om(u),2R\lambda_{\om_1}\dots\lambda_{\om_n})$. It is not difficult to check that each $g_n$ is indeed $\B(\Omega \times \X_\infty)$-measurable.
Furthermore, let
\[
 g_\infty(\om,u):= \lim_{n\to\infty} g_n(\om,u),
\]
whenever $u \in \X^\pom_\infty$ and the limit exists, and set $g_\infty(\om,u)=0$ otherwise. Then  $g_\infty$ is measurable (since the $g_n$ are), and it follows from Theorem \ref{thm:rcp}(iii) and the definition \eqref{eq:exact0} that, for $\overline{\PP}$-almost all $(\om,u)$,
\[
g_n(\om,u) \to g_\infty(\om,u) =  -\log \left( \mathbf{P}( \overline{X}^\pom_1 = u_1 | X^\pom = \Pi_\om(u) )\right) .
\]
Furthermore, by the proof of  \cite[Proposition~3.5]{FengHu09} (see also \cite[Proposition~2.3]{FalconerJin14}), there exists a constant $K > 0$ independent of $\om$ such that
\[
\int_{\mathbb{X}^\pom_\infty} \left[ \sup_{n \in \N_0} g_n(\om,u)\right] d\overline{\eta}^\pom(u) \leq H (p_{\om_1}) + K \leq \log k_{\max} + K.
\]

Now, we invoke a result of Maker \cite{Maker40}.

\begin{theorem} \label{thm:maker} Let $(X,\B,\PP,T)$ be a measure-preserving system and $(g_n)_{n \in \N}$ be a sequence of integrable functions on $(X,\B,\PP)$. If $g_n(x) \rightarrow g_\infty(x)$ for $\PP$-a.e. $x \in X$ and $\sup_{n \in \N} |g_n|$ is integrable, then for $\PP$-a.e. $x$
	\[
	\lim_{n \rightarrow +\infty} \frac{1}{n} \sum_{k=0}^{n-1} g_{n-k}(T^k x) = \overline{g}_\infty(x),
	\] where
	\[
	\overline{g}_\infty(x)= \lim_{n \rightarrow +\infty} \frac{1}{n} \sum_{k=0}^{n-1} g_{\infty}(T^k x).
	\] Furthermore, if $(X,\B,\PP,T)$ is ergodic then $\overline{g}_\infty(x)=\PP(g_\infty)$ for $\PP$-a.e. $x \in X$.
\end{theorem}
It follows from Lemma \ref{lem:ovP-ergodic}, Theorem \ref{thm:maker} and our previous discussion that if $\PP$ is ergodic, then for $\overline{\PP}$-almost every $(\om,u)$
\begin{equation} \label{eq:limit-fiber-entropy}
\lim_{n \rightarrow \infty} \frac{1}{n}\sum_{k=0}^{n-1} g_{n-k}(T^k (\om,u)) = -\int_{\Omega \times \X_\infty} \log \left( \mathbf{P}( \overline{X}^{(\om')}_{1} = u'_1 | X^{(\om')} = \Pi_{\om'}(u'))\right) d\overline{\PP}(\om',u').
\end{equation}
Let us compute the term $g_{n-k}(T^k (\om,u))$ for each $k=0,\dots,n-1$. Note that if $u \in \X^\pom_\infty$ then
\begin{align*}
g_{n-k}(T^k (\om,u)) &= -\log \left( \mathbf{P}( \overline{X}^{(T^k \om)}_1 = (T^ku)_1 | X^{(T^k \om)} \in B_{n-k}(T^k (\om,u)))\right)\\
\\
& = -\log \left( \frac{\mathbf{P}( \overline{X}^{(T^k \om)}_1 = (T^ku)_1 , X^{(T^k \om)} \in B_{n-k}(T^k (\om,u)))}{\mathbf{P}(X^{(T^k \om)} \in B_{n-k}(T^k (\om,u)))}\right)\\
\\
& = -\log \left( \frac{\mathbf{P}( \overline{X}^\pom_{k+1}=u_{k+1}, T^k X^\pom  \in B_{n-k}(T^k (\om,u)))}{\mathbf{P}(T^k X^\pom \in B_{n-k}(T^k (\om,u)))}\right)
\end{align*} where, for the last inequality, we use property (PII) of the construction of the $(\overline{X}^\pom)_{\om \in \Omega}$. Now, since the events in the last line depend only on $T^k \overline{X}^\pom$, using property (PIII)  we can rewrite the last line as
\[ -\log\left(\frac{\mathbf{P}( \overline{X}^\pom|k+1 =u|k+1\,,\, T^k X^\pom  \in B_{n-k}(T^k (\om,u)))}{\mathbf{P}(\overline{X}^\pom|k = u|k\,,\,T^k X^\pom \in B_{n-k}(T^k (\om,u)))}\right)
\] by multiplying and dividing by $\mathbf{P}(\overline{X}^\pom|k=u|k)$ inside the logarithm. Furthermore, notice that
\begin{equation}
\label{eq:equality}
\{\overline{X}^\pom|k = u|k\,,\,T^k X^\pom \in B_{n-k}(T^k (\om,u))\} = \{ \overline{X}^\pom|k = u|k\,,\,X^\pom \in B_{n}(\om,u)\}.
\end{equation} Indeed, the equality in \eqref{eq:equality} follows at once upon noticing that for all $u \in \X^\pom_\infty$
\begin{equation}\label{eq:equality1}
f^\pom_{u|k}( \Pi_{T^k \om}(T^k u)) = \Pi_\om (u)
\end{equation} and for every $v,v' \in \R$
\begin{equation} \label{eq:equality2}
|f^\pom_{u|k}(v) - f^\pom_{u|k}(v')|=\lambda_{\om_1}\cdots	\lambda_{\om_k}|v-v'|.
\end{equation} Thus, we obtain
\begin{equation} \label{eq:equality6}
g_{n-k}(T^k (\om,u)) = -\log\left(\frac{\mathbf{P}( \overline{X}^\pom|k+1 =u|k+1\,,\, X^\pom  \in B_{n}( \om,u))}{\mathbf{P}(\overline{X}^\pom|k = u|k\,,\, X^\pom \in B_{n}(\om,u))}\right)
\end{equation} so that for all $n \in \N$ the average $A_n(\om,u)=\frac{1}{n}\sum_{k=0}^{n-1} g_{n-k}(T^k (\om,u))$ telescopes to yield
\begin{align}
A_n(\om,u) &= -\frac{\log \mathbf{P}( \overline{X}^\pom|n =u|n\,,\, X^\pom  \in B_{n}( \om,u))}{n} + \frac{\log \mathbf{P}( X^\pom \in B_n(\om,u))}{n} \nonumber\\ \nonumber
\\ & = -\frac{\log \mathbf{P}( \overline{X}^\pom|n =u|n)}{n} + \frac{\log \mathbf{P}( X^\pom \in B_n(\om,u))}{n} \label{eq:equality3}
\end{align} where, for the last equality, we use that $\{\overline{X}^\pom|n = u|n\} \subseteq \{X^\pom \in B_n(\om,u)\}$, a fact which follows again from \eqref{eq:equality1} and \eqref{eq:equality2} upon noticing that the image of $\Pi_{T^n \om}$ is contained in $[-R,R]$. By property (PI) of the construction of the $(\overline{X}^\pom)_{\om \in \Omega}$, the first term in the right-hand side of \eqref{eq:equality3} converges to $\alpha_1$ as $n \rightarrow +\infty$ for $\overline{\PP}$-almost every $(\om,u)$. Therefore, by Lemma \ref{lema:exact} and \eqref{eq:limit-fiber-entropy}, we conclude that $\overline{\PP}$-almost surely
\[
\lim_{n \rightarrow +\infty} \frac{-\log \mathbf{P}( X^\pom \in B_n(\om,u))}{n} = \alpha_1-\alpha_2,
\]
where $\alpha_2$ is the right-hand side of \eqref{eq:limit-fiber-entropy}.

Now, recall that $B_n(\om,u)=B(\Pi_\om(u),2R\lambda_{\om_1}\dots\lambda_{\om_n})$ and also note that
\[
\lim_{n\to\infty}\frac{-\log(2R\lambda_{\om_1}\dots\lambda_{\om_n})}{n} = \int_\Omega -\log(\lambda_{\om'_1}) d\PP(\om')
\]
for $\PP$-almost every $\om$ by the ergodic theorem. Thus, from (PI) and the representation of $\overline{\PP}$ in \eqref{eq:exact0}, we deduce that $\eta^\pom$ is exact-dimensional for $\PP$-almost every $\om$, with  exact dimension given by \eqref{eq:value-exact-dim}. This concludes the proof.

The following corollary is a straightforward consequence of the proof of Theorem \ref{thm:exact}.

\begin{corollary} \label{cor:dimension-conservation} If $\PP$ is ergodic then the projection $\Pi_\om$ is \textit{dimension-conserving} for $\PP$-almost every $\om \in \Omega$, i.e. the measures $\overline{\eta}^\pom$ and $\eta^\pom$ are both exact-dimensional and, furthermore, for $\eta^\pom$-almost every $x \in \R$ the measure $Q^\pom_x:=Q^\pom(x,\cdot)$ is also exact-dimensional and satisfies the equality
	\[
	\textup{dim}\, \overline{\eta}^\pom = \textup{dim}\, Q^\pom_x + \textup{dim}\, \eta^\pom \left(\int_\Omega -\log(\lambda_{\om'_1}) d\PP(\om')\right).
	\]
\end{corollary}

\begin{proof}
By taking $n \rightarrow +\infty$ in \eqref{eq:equality6} we obtain for $\overline{\PP}$-almost every $(\om,u)$ that
\begin{align*} 
g_{\infty}(T^k(\om,u)) &= -\log (\mathbf{P}( \overline{X}^\pom|k+1 =u|k+1\,|\,X^\pom=\Pi_\om(u)))\\
&  + \log (\mathbf{P}( \overline{X}^\pom|k =u|k\,|\,X^\pom=\Pi_\om(u))).
\end{align*}
In particular, we see that
\[
\frac{1}{n}\sum_{k=0}^{n-1} g_\infty( T^k(\om,u)) = -  \frac{\log (\mathbf{P}( \overline{X}^\pom|n =u|n\,|\,X^\pom=\Pi_\om(u)))}{n} 
\]
which, by the ergodic theorem for the pair $(\overline{\PP},T)$, implies that for $\PP$-a.e. $\omega$ we have
\begin{equation} \label{eq:dimexactfinal}
\alpha_2=\overline{\PP}(g_\infty)=\lim_{n \rightarrow +\infty} - \frac{\log (\mathbf{P}( \overline{X}^\pom|n =u|n\,|\,X^\pom=\Pi_\om(u)))}{n}
\end{equation}
for $\overline{\eta}^\pom$-a.e. $u \in \X_\infty$. Now, let us take $\omega \in \Omega$ such that \eqref{eq:dimexactfinal} holds $\overline{\eta}^\pom$-almost surely. Using (iii) in Theorem \ref{thm:rcp}, one can show that, for $\overline{\eta}^\pom$-almost every $u$, the equality
\[
Q^\pom(\Pi_\om(u),[u|n]_\om) = \mathbf{P}( \overline{X}^\pom|n =u|n\,|\,X^\pom=\Pi_\om(u))
\]
holds for all $n \in \N$. From this and \eqref{eq:dimexactfinal} we deduce that for $\eta^\pom$-a.e. $x$ the measure $Q^\pom_x$ is exact-dimensional and, furthermore, satisfies $\textup{dim}\, Q^\pom_x  = \alpha_2$. The corollary now follows from \eqref{eq:value-exact-dim} and Remark \ref{rem:exact}.
\end{proof}

\section{Dimension, entropy and super-exponential concentration}\label{sec:result}
\label{sec:dimension}

\subsection{Preliminaries on entropy and entropy dimension}\label{secentropia}

In this section we will prove Theorem \ref{thm:sup-exp-conc}. Throughout this section, we work with a fixed model $\Sigma:= \left((\Phi^{(i)})_{i\in I},(p_i)_{i \in I},\PP\right)$ for which $\PP$ is a 
Bernoulli globally supported measure. See the footnote to the statement of Theorem \ref{thm:sup-exp-conc}.

We start by reviewing some definitions and facts related to entropy. The Shannon entropy of a probability measure $\nu$ with respect to a countable partition $\mathcal{E}$  is given by the formula
\[
H(\nu,\mathcal{E}):= - \sum_{E \in \mathcal{E}} \nu(E)\log \nu(E),
\] where the logarithm is (from now onwards) to base $2$, and $0 \log 0 = 0$. The conditional entropy with respect to the countable partition $\mathcal{F}$ is then defined as
\[
H(\nu,\mathcal{E}|\mathcal{F}):= \sum_{F \in \mathcal{F}:\nu(F)>0} \nu(F)H(\nu_F,\mathcal{E})
\] where $\nu_F:= \frac{1}{\nu(F)}\nu|_F$ is the conditional measure on $F$. For a probability vector $q=(q_1,\dots,q_k)$ we write
\[
H(q):= - \sum_{i=1}^k q_i \log q_i.
\]
Finally, we write $H_n(\nu):= \frac{1}{n}H(\nu,\mathcal{D}_{n})$ for the normalized $n$-scale entropy of $\nu$, where $\mathcal{D}_n$ is the family of $n$-dyadic intervals given by
\[
\mathcal{D}_n = \left\{ \left[\frac{j}{2^n}, \frac{j+1}{2^n}\right) : j \in \Z\right\}.
\]

Below we collect some standard properties of entropy we are to use in the sequel. We denote the total variation distance between Borel probability measures by $d_{TV}$, and recall that it is defined as
\[
d_{TV}(\nu_1,\nu_2) := \sup_{A \in \B(\R)} |\nu_1(A)-\nu_2(A)|.
\]

\begin{proposition}\label{propentro} The entropy $H$ satisfies the following properties:
\begin{enumerate}
	\item [i.] If $\nu$ is supported on $k$ elements of $\mathcal{E}$ then $H(\nu,\mathcal{E}) \leq \log k$.
	\item [ii.] If $\mathcal{E}$ refines $\mathcal{F}$ then $	H(\nu,\mathcal{E})= H(\nu,\mathcal{F}) + H(\nu,\mathcal{E}|\mathcal{F})$.
	\item [iii.] Both $H(\cdot,\mathcal{E})$ and $H(\cdot,\mathcal{E}|\mathcal{F})$ are concave.
	\item [iv.] If each element of $\mathcal{E}$ intersects at most $k$ elements of $\mathcal{F}$ and vice versa then
	\[
	|H(\nu,\mathcal{E}) - H(\nu,\mathcal{F})|=O(\log k)
	\] independently of $\nu$. In particular, if $\nu = \lambda \cdot \tilde{\nu} + x$ with $C^{-1} < \lambda < C$ then
	\[
	|H(\nu,\mathcal{D}_n) - H(\tilde{\nu},\mathcal{D}_n)| = O_C(1)
	\] independently of $n$, $\lambda$ and $x$.
	\item [v.] If $\mu,\nu \in \P([-R,R])$ then for all $n \in \N$
	\[
	H(\mu \ast \nu,\D_n) \geq H(\mu,\D_n) - O_R(1).
	\]
	\item [vi.] Given $\varepsilon \in (0,\frac{1}{2})$ there exists $\delta > 0$ such that if $\nu,\tilde{\nu}$ are probability measures with $d_{TV}(\nu,\tilde{\nu})< \delta$ then for any partition $\mathcal{E}$ with $k$ elements
	\[
	|H(\nu,\mathcal{E}) - H(\tilde{\nu},\mathcal{E})| < \varepsilon \log k + H(\varepsilon).
	\] In particular, if $\nu,\tilde{\nu} \in \mathcal{P}([-R,R])$ are such that $d_{TV}(\nu,\tilde{\nu})< \delta$, then
	\[
	|H_m(\nu) - H_m(\tilde{\nu})| < \varepsilon\left(\frac{\log 2R}{m} + 1\right) + \frac{H(\varepsilon)}{m}.
	\]
\end{enumerate}
\end{proposition}
\begin{proof}
These properties are well-known; since we were not able to find a single source for all them, we sketch the proofs. Property i. is immediate from the concavity of the logarithm, while Property iii. follows at once from the concavity of $\phi(x)=-x\log(x)$.  For ii. using that $\phi(xy) = x\phi(y)+y\phi(x)$, we have
\begin{align*}
H(\nu,\mathcal{E}) &= \sum_{F\in\mathcal{F}} \sum_{E\in\mathcal{E},E\subset F} \phi(\nu(F)\nu_F(E)) \\
&= \left(\sum_{F\in\mathcal{F}} \nu(F) \sum_{E\in\mathcal{E}} \phi(\nu_F(E))\right) + \left( \sum_{F\in\mathcal{F}} \phi(\nu(F)) \sum_{E\in\mathcal{E}} \nu_F(E) \right) \\
&=  H(\nu,\mathcal{E}|\mathcal{F})+H(\nu,\mathcal{F}) .
\end{align*}
For iv., note that it follows from i. and ii. that $|H(\nu,\mathcal{E}\vee \mathcal{F})-H(\nu,\mathcal{E})|\le \log k$, where $\mathcal{E}\vee \mathcal{F}$ is the coarsest common refinement of $\mathcal{E}$ and $\mathcal{F}$, and likewise for $\mathcal{F}$. Property v. follows from the identity $\mu \ast \nu = \int (\mu \ast \delta_y) d\nu(y)$, concavity of entropy and Property iv., since $\mu \ast \delta_y$ is a translate of $\mu$. Finally, to show Property vi. we first observe the \textit{convexity bound}
\[ \label{eq:convexitybound}
H\left( \sum_{i=1}^k q_i \mu_i, \mathcal{E}\right)\leq \sum_{i=1}^k q_iH(\mu_i,\mathcal{E}) + H(q)
\] valid for all probability vectors $q=(q_1,\dots,q_k)$ and probability measures $\mu_1,\dots,\mu_k$. Indeed, since $\phi(\sum_i x_i)\leq \sum_i \phi(x_i)$ for $x_i \in [0,1]$, and again using $\phi(xy)=x\phi(y)+y\phi(x)$,
\begin{align*}
H\left(\sum_{i=1}^k q_i \mu_i,\mathcal{E}\right) &  \leq \sum_{E \in \mathcal{E}} \sum_{i=1}^k \phi\left(q_i \mu_i(E)\right)\\
& = \sum_{i=1}^k q_i \left(\sum_{E \in \mathcal{E}} \phi(\mu_i(E))\right) + \sum_{i=1}^k \phi(q_i)\left( \sum_{E \in \mathcal{E}} \mu_i(E)\right)\\
& =\sum_{i=1}^k q_i H(\mu_i,\mathcal{E}) + H(q).
\end{align*} Now, it follows from the results in \cite[Chapter 3.7.3 and 3.8.3]{Tho} that given $\varepsilon \in (0,\frac{1}{2})$ there exists $\delta> 0$ and $\alpha \in (0,\varepsilon)$ such that if $d_{TV}(\nu,\tilde{\nu})< \delta$ then there exist probability measures $\tau,\nu_*,\tilde{\nu}_*$ with $\nu=(1-\alpha)\tau + \alpha \nu_*$ and $\tilde{\nu}=(1-\alpha)\tau + \alpha \tilde{\nu}_*$, so that by the concavity of entropy and the convexity bound \eqref{eq:convexitybound} we obtain
\[
|H(\nu,\mathcal{E})-H(\tilde{\nu},\mathcal{E})| < \alpha \log k + H(\alpha) < \varepsilon \log k + H(\varepsilon),
\] where for the last inequality we have used that $\alpha <\varepsilon< \frac{1}{2}$ and that $H$ is increasing as a function of $\alpha$ on $(0,\frac{1}{2})$.
\end{proof}

Recall that for any Borel probability measure $\nu$ on $\R$, we say that $\nu$ has \textit{entropy dimension} $\theta$ whenever
\[
\lim_{n \rightarrow +\infty} H_n(\nu)= \theta,
\]
and we denote it by $\textup{dim}_e\, \nu = \theta$.  As is well known, if $\nu$ has exact dimension $\theta$ then $\textup{dim}_e\,\nu=\theta$. Hence, we have the following immediate corollary of Theorem \ref{thm:exact}.
\begin{corollary}\label{cor:entropydim} There is a full measure set $\Omega^*\subset\Omega$ such that $\textup{dim}_e(\eta^\pom)=\alpha$ for all $\om\in\Omega^*$, where $\alpha$ is given by \eqref{eq:value-exact-dim}.

In particular, $H_n(\eta^{(\cdot)}) \overset{\PP}{\longrightarrow} \alpha$, i.e. for any given $\varepsilon > 0$ we have
\[
\lim_{n \rightarrow +\infty} \PP\left( \left\{ \om \in \Omega : |H_n(\eta^\pom)-\alpha| > \varepsilon \right\} \right) = 0.
\]	
\end{corollary}

\subsection{A theorem on entropy growth}
\label{subsec:entropy-growth}
Theorem \ref{thm:sup-exp-conc} is derived from a result about the entropy of finite approximations of $\eta^\pom$. Namely, if we define $\nu^{(\om,n)}$ as the projection of $\oeta^\pom$ via the $n$-truncated coding map $\Pi^{(n)}_\om$ given by
\[
\Pi^{(n)}_\om := \sum_{k=1}^n \left(\prod_{j=1}^{k-1}\lambda_{\om_j}\right)t^{(\om_k)}_{u_k},
\]
then we will show that Theorem \ref{thm:sup-exp-conc} is a consequence of Theorem \ref{teo2} below. For simplicity of notation, we assume that $R$ is chosen so large that $\tfrac12 < 2R\lam_{\min}$. Given $n \in \N$ we shall write $\ell^\pom_n$ for the unique natural number such that
\begin{equation}\label{ell}
2R \prod_{i=1}^{\ell^\pom_n} \lambda_{\om_i} \leq 2^{-n} < 2R \prod_{i=1}^{\ell^\pom_n -1} \lambda_{\om_i}.
\end{equation}
We will often write $n'$ or $n'(\om)$ for $\ell^{\pom}_n$. Recall that we have fixed a model $\Sigma$ with a Bernoulli selection measure $\PP$.

\begin{theorem}\label{teo2} We have the following implication:
\[
	\dim(\Sigma) < 1 \Longrightarrow \frac{H(\nu^{(\cdot,n'(\cdot))},\mathcal{D}_{(q+1)n}|\mathcal{D}_n)}{n} \overset{\PP}{\longrightarrow}0 \text{ for all $q \in \N$,}
\]
where, for each $n \in \N$ and $\omega \in \Omega$, we abbreviate $n'(\om):=\ell^\pom_n$.
\end{theorem}
This result extends \cite[Theorem 1.3]{Hochman14} to our setting. We will see how Theorem \ref{thm:sup-exp-conc} is derived from Theorem \ref{teo2} in Section \ref{subsec:endproof} below (this is analogous to the deduction of \cite[Theorem 1.1]{Hochman14}  from \cite[Theorem 1.3]{Hochman14}).

\subsection{Component measures and the inverse theorem} \label{subsec:component-measures}

Given $x \in \R$ let $D_n(x)$ denote the unique $n$-level dyadic interval containing $x$. For $D \in \mathcal{D}_n$ let $T_D : \R \to \R$ denote the unique homothety mapping $D$ to the interval $[0,1)$. Recall that if $\nu$ is a probability measure on $\R$ then $T_D\, \nu$ is the push-forward of $\nu$ through $T_D$.

\begin{definition} For a probability measure $\nu$ on $\R$ and a dyadic cell $D$ with $\nu(D) > 0$ we define:
	\begin{enumerate}
		\item [$\bullet$] The \textit{raw $D$-component} of $\nu$ as $\nu_D := \frac{1}{\nu(D)} \nu|_D$.
		\item [$\bullet$] The \textit{rescaled $D$-component} of $\nu$ as $\nu^D:= \frac{1}{\nu(D)}T_D\,\nu|_D$.
	\end{enumerate}
	Also, for $x \in \R$ such that $\nu(D_n(x))> 0$ we write $\nu_{x,n}:= \nu_{D_n(x)}$ and $\nu^{x,n}:= \nu^{D_n(x)}$. We call these the $n$-level components of $\nu$.
\end{definition}
Finally, for any bounded measurable function $f: \mathcal{P}(\R) \to \R$ and finite $J \subseteq \N$ we define
\begin{equation} \label{cintegration}
\mathbf{E}_J(f(\nu_{x,i})):= \frac{1}{|J|} \sum_{j \in J} \int_{\R} f(\nu_{y,j})d\nu(y),
\end{equation} i.e. the expectation of the random variable $f(\nu_{x,i})$ when $i$ is chosen at random uniformly in $J$ and $x$ is chosen at random according to $\nu$. Likewise, we write $\mathbf{P}_J(\nu_{x,i}\in A)$ for $\mathbf{E}_J(\mathbf{1}_A(\nu_{x,i}))$.

Notice that for any Borel set $A \subseteq \R$ one has the identity
\begin{equation}
\label{eq:descomp}
\nu(A)=\mathbf{E}_{n}(\nu_{x,i}(A)).
\end{equation}
We also define the analogue of \eqref{cintegration} for the rescaled component measures $\nu^{x,i}$ in the obvious manner. From this definition one obtains the identity
\begin{equation}
\label{eq:identity}
\mathbf{E}_n ( H_m( \nu^{x,i})) = \frac{1}{m} H( \nu, \D_{n+m}|\D_n)
\end{equation}valid for all $n,m \in \N$. See \cite[Section 3.2]{Hochman14}. Furthermore, one also has the following result, proved in \cite[Lemma 3.4]{Hochman14} .
\begin{lemma}\label{lemaentro}
Given $R > 1$, for every $\nu \in \mathcal{P}([-R,R])$ and $m < n \in \N$
\[
H_n(\nu) = \mathbf{E}_{\{1,\dots,n\}} \left( H_m\left(\nu^{x,i}\right) \right) + O\left(\frac{m}{n} + \frac{\log R}{n} \right)
\]
\end{lemma}

We can now state the inverse theorem for the entropy of convolutions which constitutes the core of our approach. It is originally featured in \cite[Theorem 2.8]{Hochman14} (see also the remark after \cite[Theorem 2.9]{Hochman14}).
\begin{theorem}\label{thm:inverse} Given $\varepsilon,R > 0$ and $m \in \N$, there exists $\delta = \delta(\varepsilon,R,m) > 0$ such that, for every $n \geq n(\varepsilon,\delta,R,m)$ and every $\tau \in \mathcal{P}([-R, R])$, if
	\[
	\mathbf{P}_{\{1,\dots,n\}}( H_m(\tau^{x,i}) < 1 - \varepsilon) > 1 - \varepsilon,
	\]
	then for every $\nu \in \mathcal{P}([-R,R])$ one has
	\[
	H_n(\nu) > \varepsilon \Longrightarrow H_n(\nu \ast \tau ) \geq H_n (\tau) + \delta.
	\]
\end{theorem}

\subsection{Uniform entropy dimension of $\eta^\pom$} \label{subsec:uniform-entropy-dim}

Our next objective is to show the following result, which we shall use to prove Theorem \ref{teo2}.
\begin{theorem}\label{thm:ued}  There exists a full $\PP$-measure set $\Omega' \subseteq \Omega$ such that $\eta^\pom$ has \textit{uniform entropy dimension} $\alpha=\dim(\Sigma)$ for any $\omega \in \Omega'$, i.e. given any $\varepsilon > 0$, for every $m \in \N$ sufficiently large (depending only on $\varepsilon$ and $\om$)
	\[
	\liminf_{n \rightarrow +\infty} \mathbf{P}_{\{1,\dots,n\}} \left( \left| H_m\left(\eta^{\pom,x,i}\right) - \alpha\right| < \varepsilon \right) > 1 - \varepsilon,
	\] where we write $\eta^{\pom,x,i}:= \left(\eta^\pom\right)^{x,i}$.
\end{theorem}
Again, this generalizes \cite[Proposition 5.2]{Hochman14} to our random setting. As discussed in \cite[Section 5.1]{Hochman14}, the notion of uniform entropy dimension (UED) is stronger than the one given in Section \ref{sec:result}.

We begin by recalling that degenerate systems generate measures supported on a single atom, and so Theorem \ref{thm:ued} is trivial in this case.  Therefore, for the rest of this section we make the following assumptions.

\begin{assump}\label{assump:1} There exists $i_0 \in I$ such that:
\begin{enumerate}
	\item [(i).] $k_{i_0}\ge 2$ and $t_j^{(i_0)}\neq t_{j'}^{(i_0)}$ for some $j,j'$, i.e. the model is non-degenerate.
	\item [(ii).] $\PP(\{ \om \in \Omega : \om_1 = i_0\}) > 0$.
\end{enumerate}
\end{assump}
Nonetheless, to avoid any possible confusion, in the sequel we will indicate it explicitly whenever these assumptions are needed.

We begin the proof of Theorem \ref{thm:ued} with a few preliminary lemmas. In the setting of \cite[Proposition 5.2]{Hochman14}, a key fact is that a self-similar measure is either continuous or supported on a single atom. We need an analog of this fact, but the situation is more involved since, whenever $k_i=1$ for some $i$, there will be some $\om$ for which $\eta^\pom$ is indeed atomic. Moreover, even if we know that $\eta^\pom$ is continuous $\PP$-almost surely, in general there will be no modulus of continuity valid $\PP$-almost everywhere. The next lemma will help us cope with this lack of uniformity.

We say that a measure $\gamma$ is \emph{$(C,\rho)$-Frostman} if $\gamma(B(x,r)) \leq C r^\rho$ for all $x\in \R$ and $r>0$. The following result asserts that, for some fixed $\rho>0$, the measures $\eta^{(T^k\om)}$ are often $(C,\rho)$-Frostman.

\begin{lemma}\label{lema:ued1} If $\PP$ verifies \mbox{Assumptions \ref{assump:1}}, then there exist $\rho>0$ and a full $\PP$-measure set $\Omega^{(1)} \subseteq \Omega$, such that for any $\delta > 0$ there is a constant $C_\delta > 0$ satisfying
	\[
	\liminf_{n \rightarrow +\infty} \frac{1}{n}\left|\{ k \in \{1,\ldots,n\} : \eta^{(T^k \om)} \text{ is $(C_\delta,\rho)$-Frostman} \}\right|\geq 1-\delta
	\]
	for every $\om \in \Omega^{(1)}$.
\end{lemma}

\begin{proof} By the ergodic theorem, our task is to show the existence of $\rho>0$ such that, given $\delta>0$, there is $C_\delta>0$ satisfying
\[
\PP\left(\{\om: \eta^\pom \text{ is $(C_\delta,\rho)$-Frostman}\}\right) > 1-\delta.
\]
In turn, for this it is enough to show that there is $\rho>0$ such that for $\PP$-almost all $\om$ there is $r_0=r_0(\om)>0$ satisfying
\begin{equation} \label{eq:Frostman-small-scales}
\eta^\pom(B(x,r)) \le  C_0 r^{\rho} \text{ for all } x\in\R ,r\in (0,r_0),
\end{equation}
where $C_0>0$ depends only on the model.

First, notice that there exist $\ell\in\N$, $\eps>0$, such that whenever $\om_1=i_0$ (where $i_0$ is from Assumptions \ref{assump:1}), then there are $u,v \in \X^\pom_\ell$ such that $B^\pom_u$ and $B^\pom_v$ are $(2\eps)$-separated; recall \eqref{eq:def-symbolic-interval}. Indeed, pick $u,v$ starting with $j,j'$ respectively, where $t_j^{i_0}\neq t_{j'}^{i_0}$, and otherwise having equal entries.  Then
\begin{align*}
\dist(B^\pom_u, B^\pom_v) &\ge | f_u^\pom(0)-f_v^\pom(0)| -  2 R \lambda_{\om_1}\cdots \lambda_{\om_\ell} \\
&\ge |t_j^{i_0}-t_{j'}^{i_0}| - 2 R \lam_{\max}^\ell.
\end{align*}
Thus we can ensure that $\dist(B^\pom_u, B^\pom_v)\ge 2\eps$ by taking $\ell$ large enough.

Now, given $\om \in\Omega$ containing infinitely many $i_0$'s (which is the case $\PP$-almost surely) let us define a subsequence $(n_j)_{j \in \N_0}=(n_j(\om))_{j \in \N_0}$ as follows: set $n_0 := 0$ and then, having defined $n_j$ for $j \in \N_0$, set
\[
n_{j+1} = \min\{ n\ge n_j+\ell: \omega_n=i_0\}.
\]
Since $i_0$ appears in $(\om_1,\ldots,\om_{n_j})$ at most $j\ell$ times, it follows from the ergodic theorem that for $\PP$-a.e. $\om$,
\[
\limsup_{j\to\infty} \frac{n_j}{j} \le \frac{\ell}{\PP(\{\om:\om_1=i_0\})} =: \frac{C'}{2} <\infty.
\]
This implies that for $\PP$-a.e. $\om$ there exists $J(\om)\in\N$ such that for any $j\ge J(\om)$,
\begin{equation} \label{eq:not-many-big-jumps}
|\{ k\in\{1,\ldots, j-1\}: n_{k+1}-n_k > 2 C'\}| \le \frac{j}{2}.
\end{equation}
Let \[
	\phi^\pom(r):= \sup_{x \in \R} \eta^\pom(B(x,r)).
	\]
Fix $j\ge 1$. For simplicity, let us write $\omega^j := T^{n_j-1} \om$, and note that $\om^j_1=i_0$. Then, if $u,v \in \X^{(\omega^j)}_{n_{j+1}-n_j}$ are such that $\dist(B^{(\omega^j)}_u,B^{(\omega^j)}_v) > 2\eps$ (such $u,v$ exist since $n_{j+1}-n_j \geq \ell$ and $\omega^j_1=i_0$) and we take $r \in (0,\eps]$, for every $x \in \R$ we have either $B(x,r)\cap B^{(\omega^j)}_u = \emptyset$ or $B(x,r) \cap B^{(\omega^j)}_v = \emptyset$. Assuming without loss of generality that we are in the second case, we can use the self-similarity relation \eqref{eq:dssr} to estimate
	\begin{align}
	\eta^{(\omega^j)}(B(x,r)) &= \sum_{e \in \X^{(\omega^j)}_{n_{j+1}-n_j}} p^{(\omega^j)}_e \cdot f^{(\omega^j)}_e \eta^{(\omega^{j+1})} (B(x,r)) \nonumber \\
	&= \sum_{e \neq v} p^{(\omega^j)}_e \cdot \eta^{(\omega^{j+1})} \left( \left(f^{(\omega^j)}_e\right)^{-1}(B(x,r))\right) \nonumber \\
	& \leq \sum_{e \neq v} p^{(\omega^j)}_e \cdot \phi^{(\omega^{j+1})}\left(\frac{r}{\lambda_{\min}^{n_{j+1}-n_j}}\right) \nonumber \\
	&  \leq (1 - p_{\min}^{n_{j+1}-n_j}) \phi^{(\omega^{j+1})}\left(\frac{r}{\lambda_{\min}^{n_{j+1}-n_j}}\right)	\label{eq:sep5}
	\end{align}
	where $p_{\min} := \min \{ p^{(i)}_j : j=1,\dots,k_i, i\in I \} > 0$ and $\lambda_{\min}=\min\{ \lambda_i,i\in I \}>0$. A similar but easier argument yields
\[
\eta^\pom(B(x,r)) \le \phi^{(\omega^1)}\left(\frac{r}{\lam_{\min}^{n_1}}\right).
\]
Starting with this bound,  iterating \eqref{eq:sep5}, and recalling \eqref{eq:not-many-big-jumps}, we conclude that if $j\ge J(\om)$ then
    \[
    \phi^{(\omega)}(\eps \cdot \lambda_{\min}^{n_j}) \le (1-p_{\min}^{2C'})^{j/2}.
 \]
 Finally, pick $J'(\om)\ge J(\om)$ such that $\frac{n_{j+1}}{j} \leq C'$ for all $j \geq J'(\om)$. If $0<r \le r_0:= \eps \cdot \lambda_{\min}^{n_{J'(\om)}}$, we can find $j'\ge J(\om)$ such that $\eps \cdot \lambda_{\min}^{n_{j+1}} < r\le \eps \cdot \lambda_{\min}^{n_j}$, so that
 \[
 \eta^\pom(B(x,r)) \le \phi^\pom( \eps \cdot \lambda_{\min}^{n_j}) \le (1-p_{\min}^{2C'})^{j/2} \le \eps^{-\rho} r^\rho,
 \]
 where $\rho=\log(1-p_{\min}^{2C'})/(2 C' \log\lam_{\min})$. We have verified that \eqref{eq:Frostman-small-scales} holds, as desired.
\end{proof}

One has the analogous result to Lemma \ref{lema:ued1} for the relative frequencies of the sets where $H_m(\eta^{(T^k \om)}) > \alpha - \delta $ (where we recall that $\alpha=\dim(\Sigma)$).

\begin{lemma}\label{lema:ued2} There exists a full $\PP$-measure set $\Omega^{(2)} \subseteq \Omega$ such that for any $\delta > 0$ there is $m_\delta \in \N$ satisfying
	\[
	\liminf_{n \rightarrow +\infty} \frac{1}{n}\left| k\in \{1,\dots,n\}:H_m(\eta^{(T^k \om)}) > \alpha - \delta  \right| \geq 1-\delta
	\] for all $m \geq m_\delta$ and every $\om \in \Omega^{(2)}$.
\end{lemma}

\begin{proof} By Corollary \ref{cor:entropydim}, given $\delta>0$ there is $m_\delta \in \N$ such that
	\[
	 \inf_{m \geq m_\delta} \PP\left( \left\{ \om \in \Omega : H_m(\eta^\pom) > \alpha - \delta\right\}\right)  > 1-\delta.
	\]
The lemma now follows from the ergodic theorem applied to the indicator function of the event $\{\om:H_m(\eta^\pom) > \alpha - \delta\}$.
\end{proof}

Next, let us fix $\om \in \Omega$ and for $n > k$ and $D\in\D_k$ define the $n$-truncated $k$-components of $\eta^\pom$ as
\[
\eta^{\pom}_{n,[D]} = \frac{1}{Z_{D,n}} \sum_{u \in \X^\pom_{n,[D]}} p^\pom_u \cdot f^\pom_u \eta^{(T^n \om)},
\] where
\[
\X^\pom_{n,[D]}= \{ u \in \X^\pom_n : B^\pom_u \subseteq D\},
\]
and $Z_{D,n}$ is a normalizing constant making $\eta^\pom_{n,[D]}$ a probability measure. We also write
\[
\eta^{\pom}_{n,[x,k]} = \eta^\pom_{n,D_k(x)}
\]
for simplicity. Put into words, the measure $\eta^{\pom}_{n,[x,k]}$ differs from $\eta^\pom_{x,k}$ in that, instead of restricting
\[
\eta^\pom = \sum_{u \in \X^{(\om)}_{n}} p_u^\pom \cdot f^\pom_u \eta^{(T^n \om)}
\]
to $D_k(x)$, we exclude all terms whose support is not entirely contained in $D_k(x)$.

Our next goal is to show that, for a suitably chosen $k$, and for $N$ large but fixed independent of $k$, the measures $\eta^\pom_{k+N,[x,i]}$ and $\eta^\pom_{x,i}$ are  close in the total variation distance (in particular, they have close entropies) with large $\mathbf{P}_k$-probability. See Proposition \ref{prop:ued3} for the precise statement. This will allow us to eventually replace $\eta^\pom_{x,k}$ by $\eta^\pom_{k+N,[x,k]}$, the point being that the latter is a convex combination of measures of the form $f^\pom_u \eta^{(T^n \om)}$, and hence we can estimate its entropy from below. This step is carried over in Corollary \ref{cor:ued4}.

Recall the coefficient $\ell^\pom_k$ defined in Section \ref{sec:result}, i.e. the integer $\ell_k$ satisfying
\[
2R \lambda_{\om_1}\dots\lambda_{\om_{\ell_k}} \approx 2^{-k}.
\]
In the course of this section we often denote this number by $k'(\om)$ or $k'$ if the dependence on $\om$ is clear from context.

\begin{proposition} \label{prop:ued3}
	For any $\varepsilon > 0$ and choice of constants $C,\rho > 0$, there exists $N_{\varepsilon}=N(\varepsilon,C,\rho) \in \N$ such that the following holds: if $\om \in \Omega$, $k \in \N$ are such that $\eta^{(T^{k'}\om)}$ is $(C,\rho)$-Frostman, and we take $n = k'(\om) + N_\varepsilon$, then
\[
\mathbf{P}_{k} \left( d_{TV}\left(\eta^\pom_{x,i}, \eta^\pom_{n,[x,i]}\right) < \varepsilon \right) > 1 - \varepsilon.
\]
\end{proposition}

\begin{proof} For simplicity, we suppress the dependence of $N_\varepsilon$ on $C$ and $\rho$ from the notation, and fix $\om, k$ as in the statement for the rest of the proof. Given $\varepsilon > 0$, let us choose $\delta \in (0,\tfrac{1}{2})$ such that any interval of length smaller than $2R\delta$ has $\eta^{(T^{k'} \om)}$-measure less than $\frac{\varepsilon^2}{8}$ (which is possible by the Frostman assumption), and pick $N_\varepsilon \in \N$ such that $\lambda_{\max}^{N_\varepsilon} < \frac{\delta}{2}$.
	
For $u \in \X^\pom_{{k'}}$ consider those $v \in \X^{(T^{{k'}} \om)}_{N_\varepsilon}$ such that $B_{uv}^\pom$ is not entirely contained in some $D\in\D_k$. Then the measure $f^\pom_{uv} \eta^{(T^{n} \om)}$ must be supported on an interval $J$ of length $2R\delta \lambda_{\om_1}\dots \lambda_{\om_{k'}} \leq \delta 2^{-k}$ centered at an endpoint of an element of $\D_k$. Since $f^\pom_u \eta^{(T^{k'} \om)}$ is supported on an interval of length less than $2^{-k}$, it follows that $f^\pom_u \eta^{(T^{k'} \om)}$ can give positive mass to at most two of such intervals $J$. Moreover, by choice of $\delta$ we have that
$f^\pom_u \eta^{(T^{k'} \om)}(J) < \frac{\varepsilon^2}{8}$ for any of them. Using the self-similarity relation \eqref{eq:dssr} applied to $\eta^{(T^{k'} \om)}$, we conclude that for each $u \in \X^\pom_{{k'}}$
\[
\sum_{v : uv \notin \X^\pom_{n}[\D_k]} p^{(T^{k'} \om)}_v < \frac{\varepsilon^2}{4},
\] where
\[
\X^\pom_{n}[\D_k]:= \bigcup_{D \in \D_k} \X^\pom_{n,[D]} = \{ z\in \X^\pom_n: B_z^\pom \subseteq D \text{ for some } D\in \D_k\}.
\]
Therefore, it follows that also
\[
\sum_{z \notin \X^\pom_{n}[\D_k]} p^\pom_z < \frac{\varepsilon^2}{4}.
\]
On the other hand, by Markov's inequality we obtain
\begin{align*}
\mathbf{P}_{k} \left( d_{TV}\left(\eta^\pom_{x,i}, \eta^\pom_{n,[x,i]}\right) \geq \varepsilon \right) &\leq \frac{1}{\eps} \, \mathbf{E}_{k}\left(d_{TV}\left(\eta^\pom_{x,i}, \eta^\pom_{n,[x,i]}\right)\right) \\
&\leq \frac{1}{\varepsilon}\sum_{D \in \mathcal{D}_k}  d_{TV}\left(\eta^\pom_D, \eta^\pom_{n,[D]}\right) \eta^\pom(D).
\end{align*}

A simple calculation using the definition of $d_{TV}$ shows that
\[
d_{TV}\left(\eta^\pom_D, \eta^\pom_{n,[D]}\right) \leq \frac{2}{\eta^\pom(D)}
\sum_{z \in \X^\pom_{n,(D)}} p^\pom_z,
\] where
\[
\X^\pom_{n,(D)}= \{ z \in \X^\pom_n : B^\pom_z \cap D \neq \emptyset, B^\pom_z \nsubseteq D \},
\]
which yields
\[
\mathbf{P}_{k} \left( d_{TV}\left(\eta^\pom_{x,i}, \eta^\pom_{n,[x,i]}\right) \geq \varepsilon \right) \leq \frac{2}{\varepsilon} \sum_{D \in \D_k} \sum_{z \in \X^\pom_{n,(D)}} p^\pom_z \leq \frac{4}{\varepsilon}\sum_{z \notin \X^\pom_{n}[\D_k]} p^\pom_z < \varepsilon
\] and thus concludes the proof.
\end{proof}

Proposition \ref{prop:ued3} has the following important corollary.

\begin{corollary}\label{cor:ued4} For each $\varepsilon > 0$ and choice of $C,\rho > 0$ there exist positive integers $m'_\varepsilon$ and $N'_\varepsilon$ (depending on $\varepsilon,C$ and $\rho$) such that the following holds: if $m \geq m'_\varepsilon$ then, for any $k \in \N$ such that $\eta^{(T^{k'(\om)}\om)}$ is $(C,\rho)$-Frostman and such that $n:= k'(\om) + N'_\varepsilon$ satisfies $H_m(\eta^{(T^n\om)})>\alpha-\tfrac{\eps}{4}$, one has
\[
	\mathbf{P}_k \left( H_{m} \left( \eta^{\pom,x,i}\right) > \alpha - \varepsilon \right) > 1 - \varepsilon.
\]
\end{corollary}

\begin{proof} Given $\varepsilon > 0$, choose $0 < \varepsilon' < \varepsilon$ sufficiently small so that $d_{TV}(\nu,\nu') < \varepsilon'$ implies that $|H_m(\nu) - H_m(\nu')| < \frac{\varepsilon}{2}$ for every $m \in \N$ and $\nu,\nu' \in \mathcal{P}([-R,R])$ (which exists by Proposition \ref{propentro}(\textup{vi})) and set $N'_\varepsilon:=N_{\varepsilon'}$ where $N_{\varepsilon'}$ is as in Proposition \ref{prop:ued3}.
	
Now, fix $k$ such that $\eta^{(T^{k'}\om)}$ is $(C,\rho)$-Frostman, and such that $n = k'(\om) + N'_{\varepsilon}$ satisfies
\[
H_m(\eta^{(T^n\om)})>\alpha-\frac{\eps}{4}.
\]
By Proposition \ref{prop:ued3} and choice of $\varepsilon'$, to conclude it will suffice to show that
\[
H_{m}( T_D\, \eta^\pom_{n,[x,k]}) > \alpha - \frac{\varepsilon}{2}
\]
for any $D \in \D_k$ and $m$ large enough. But since $\eta^\pom_{n,[x,k]}$ is a convex combination of $f^\pom_{u} \eta^{(T^{n}\om)}$ for $u \in \X^\pom_{n}$, by the concavity of entropy we have that
\begin{align*}
H_{m}( T_D\, \eta^\pom_{n,[x,k]}) &\geq \frac{1}{Z_{n,D}} \sum_{u \in \X^\pom_{n,[D]}} p^\pom_{u} \cdot H_{m} ( T_D\,f^\pom_{u} \eta^{(T^{n}\om)} ) \\
& \geq \left(\frac{1}{Z_{n,D}} \sum_{u \in \X^\pom_{n,[D]}} p^\pom_{u} \cdot H_{m} (\lambda_{\om_{{k'}+1}}\dots \lambda_{\om_{n}} \cdot \eta^{(T^{n}\om)})\right) - O\left(\frac{1}{m}\right),
\end{align*} where, for the second inequality, we have used Proposition \ref{propentro}(\textup{iv}) and the fact that $2^{-k} \approx 2R \lambda_{\om_1}\dots\lambda_{\om_{k'}}$. Thus, if we choose $m$ sufficiently large so as to guarantee that
\[
H_{m}(\kappa \cdot \mu) - O\left(\frac{1}{m}\right) >  \alpha - \frac{\varepsilon}{2}
\]
for all $ \lambda_{\min}^{N'_{\varepsilon}} < \kappa < 1$ and all probability measures $\mu$ supported on $[-R,R]$ with $H_m(\mu)>\alpha-\varepsilon/4$ (possible by Proposition \ref{propentro}(\textup{iv})), our choice of $k$ and $n$ yields the result.
\end{proof}

We now finish the proof of Theorem \ref{thm:ued}.

\begin{proof}[Proof of Theorem \ref{thm:ued}]
Consider the full $\PP$-measure set $\Omega' := \Omega^* \cap \Omega^{(1)} \cap \Omega^{(2)}$ given by Corollary \ref{cor:entropydim} and Lemmas \ref{lema:ued1}-\ref{lema:ued2}, take $\varepsilon \in (0,1)$ and fix an auxiliary $\varepsilon' \in (0,\frac{\varepsilon}{2})$. Let $\rho$ be the number given by Lemma \ref{lema:ued1} and,  given $\omega \in \Omega'$, choose $C_{\varepsilon'}$ as in Lemma \ref{lema:ued1} and $m''_{\varepsilon'}:= \max\{m_{\varepsilon'},m'_{\varepsilon'}\}$ where $m_{\varepsilon'} \in \N$ is as in Lemma \ref{lema:ued2} and $m'_{\varepsilon'}$ as in Corollary \ref{cor:ued4}. Then, for any $n \in \N$ and $m \geq m''_{\varepsilon'}$
\[
\mathbf{P}_{\{1,\dots,n\}} \left( H_m\left( \eta^{\pom,x,i}\right) > \alpha - \varepsilon' \right) \geq \frac{|\Theta^\pom \cap \{1,\dots,n\}|}{n} (1-\varepsilon'),
\] where
\[
\Theta^\pom := \left\{ k \in \N : \eta^{(T^{k'}\om)} \text{ is $(C,\rho)$-Frostman  and } H_m
\left(\eta^{(T^{k'+N'_{\eps'}}(\om))}\right)>\alpha-\tfrac{\eps'}{4}\right\}.
\]
Using Lemmas \ref{lema:ued1}-\ref{lema:ued2} and our choice of parameters, a straightforward calculation shows that
\[
\liminf_{n \rightarrow +\infty} \frac{|\Theta^\pom \cap \{1,\dots,n\}|}{n} \geq 1 - 2(\log \lambda_{\min}^{-1})\varepsilon'.
\]
We deduce that
\begin{equation} \label{eq:ued1}
\liminf_{n\rightarrow +\infty} \mathbf{P}_{\{1,\dots,n\}} \left( H_m\left( \eta^{\pom,x,i}\right) > \alpha - \varepsilon' \right) \geq  1-\eps'',
\end{equation}
where $\eps''=\eps''(\eps')$ goes to $0$ with $\eps'$.

On the other hand, since for $n$ sufficiently large (depending on $\om$) we have $|H_n(\eta^\pom) - \alpha| < \frac{\varepsilon'}{2}$ by Corollary \ref{cor:entropydim}, it follows from Lemma \ref{lemaentro} that
\[
\left|\mathbf{E}_{\{1,\dots,n\}}\left( H_m\left( \eta^{\pom,x,i}\right)\right) - \alpha\right| < \frac{\varepsilon'}{2}
\]
for all $n$ sufficiently large in terms of $\om$ and $m$. Since $H_m\left( \eta^{\pom,x,i}\right) \geq 0$, combining this with \eqref{eq:ued1} we reach the conclusion of the theorem.
\end{proof}

We record the following immediate corollary; this is the statement that will get used in the proof of Theorem \ref{teo2}.
\begin{corollary} \label{cor:uedp}For any $\varepsilon > 0$ we have
		\[
		\lim_{m \rightarrow +\infty} \PP \left( \left\{ \omega \in \Omega : \liminf_{n \rightarrow +\infty} \mathbf{P}_{\{1,\dots,n\}} \left( \left| H_m\left(\eta^{\pom,x,i}\right) - \alpha\right| < \varepsilon \right) > 1 - \varepsilon \right\}\right) = 1.
		\]
\end{corollary}

\subsection{Proof of Theorem \ref{teo2}}

We need one final proposition before we can conclude the proof of Theorem \ref{teo2}. In order to state it, observe that the measure $\eta^\pom$ can be decomposed for every $l \in \N$ as
\[
\eta^\pom = \nu^{(\om,l)} \ast \tau^{(\om,l)}
\]
where $\nu^{(\om,l)}$ is the $l$-level discrete approximation of $\eta^\pom$ defined in \S\ref{subsec:entropy-growth}, i.e. the projection of $\oeta^\pom$ under the $l$-truncated coding map, and
\[
\tau^{(\om,l)} := \lambda_{\om_1}\dots\lambda_{\om_l} \cdot \eta^{(T^l \om)}.
\]
We will study this decomposition with $l=n'=\ell^\pom_n$.

\begin{proposition}  \label{prop:intermediate-teo2}
\[
\lim_{n\rightarrow +\infty} \PP(\mathcal{G}_{\eps,\delta,q}^{(n)}) =1
\]
for any $\varepsilon,\delta > 0$ and $q\in\N$,  where
\[
\mathcal{G}_{\eps,\delta,q}^{(n)} = \left\{ \omega : \mathbf{P}_{n} \left( \frac{1}{qn} \left| H( \nu^{(\om,n')}_{x,i} \ast \tau^{(\om,n')},\D_{(q+1)n}) - H(\tau^{(\om,n')},\D_{(q+1)n})\right| < \delta \right) > 1-\varepsilon\right\}.
\]
\end{proposition}
\begin{proof}
Fix $q \in \N$ and consider the set $\Omega^*$ from Corollary \ref{cor:entropydim}. Now, take $\omega \in \Omega^*$ and observe that for any $n \in \N$ we have the identity
\[
\frac{1}{(q+1)n} H(\eta^\pom,\mathcal{D}_{(q+1)n}) = \frac{1}{q+1}\left(\frac{1}{n}H(\eta^\pom,\D_{n})\right) + \frac{q}{q+1}\left(\frac{1}{qn}H(\eta^\pom,\D_{(q+1)n}|\D_n)\right).
\]
Since $\om \in \Omega^*$, both the left-hand side above and the term $\frac{1}{n}H(\eta^\pom,\D_{n})$ converge to $\alpha$ as $n \rightarrow +\infty$.
It follows that
\begin{equation}
\label{eq:proof1}
\lim_{n \rightarrow +\infty} \frac{1}{qn}H(\eta^\pom,\D_{(q+1)n}|\D_n) = \alpha
\end{equation}as well.

By \eqref{eq:descomp} and the linearity of the convolution, for any Borel set $A \subseteq \R$ and $n \in \N$ we have
\[
\nu^{(\om,l)} \ast \tau^{(\om,l)}(A)= \mathbf{E}_{n}(\nu^{(\om,l)}_{x,i} \ast \tau^{(\om,l)}(A)).
\] Thus, if for $n \in \N$ we define $n':=\ell^\pom_n$ as in \eqref{ell}, the concavity of conditional entropy yields that
\begin{align}
\label{eq:hcond}
H(\eta^\pom,\D_{(q+1)n}|\D_n) &= H(\nu^{(\om,n')} \ast \tau^{(\om,n')},\D_{(q+1)n}|\D_n) \\
&\geq \mathbf{E}_{n}\left( H( \nu^{(\om,n')}_{x,i} \ast \tau^{(\om,n')},\D_{(q+1)n}|\D_n)\right).\nonumber
\end{align}
Moreover, since each of the measures $\nu^{(\om,n')}_{x,i} \ast \tau^{(\om,n')}$ is supported on an interval of length at most $4R\lambda_{\om_1}\dots\lambda_{\om_{n'}}$, this support can intersect at most three elements of $\D_n$, so that
\[
|H( \nu^{(\om,n')}_{x,i} \ast \tau^{(\om,n')},\D_{(q+1)n}|\D_n) - H( \nu^{(\om,n')}_{x,i} \ast \tau^{(\om,n')},\D_{(q+1)n})| =  O(1).
\]
Hence, from this and \eqref{eq:hcond} we obtain
\[
H(\eta^\pom,\D_{(q+1)n}|\D_n) \geq \mathbf{E}_{n}\left( H( \nu^{(\om,n')}_{x,i} \ast \tau^{(\om,n')},\D_{(q+1)n})\right) - O(1)
\] which, by \eqref{eq:proof1}, implies that
\[
\limsup_{n \rightarrow +\infty} \frac{1}{qn}\mathbf{E}_{n}\left( H( \nu^{(\om,n')}_{x,i} \ast \tau^{(\om,n')},\D_{(q+1)n})\right) \leq \alpha.
\] In particular, we obtain that for any $\delta > 0$
\begin{equation}
\label{eq:proof2}
\lim_{n \rightarrow +\infty} \PP\left( \left\{ \omega \in \Omega : \frac{1}{qn}\mathbf{E}_{n}\left( H( \nu^{(\om,n')}_{x,i} \ast \tau^{(\om,n')},\D_{(q+1)n})\right) < \alpha+\delta \right\}\right)=1.
\end{equation}
Now, on the other hand, observe that for each $\delta > 0$ we have
\begin{equation}
\label{eq:proof3}
\lim_{n \rightarrow +\infty} \PP \left( \left\{ \om \in \Omega : \left| \frac{1}{qn}H(\tau^{(\om,n')},\D_{(q+1)n}) - \alpha\right| < \delta\right\}\right) = 1
\end{equation}
since:
\begin{enumerate}
	\item [i.] For any $\omega \in \Omega$ the relation
	\begin{equation}
	\label{eq:equivtau}
	H(\tau^{(\om,n')},\D_{(q+1)n}) = H( \eta^{(T^{n'} \om)},\D_{qn}) + O(1)
	\end{equation} holds due to Proposition \ref{propentro}(iv) and the fact that
	\begin{equation} \label{eq:compn}
	2R \cdot 2^{-qn} < \frac{2^{-(q+1)n}}{\lambda_{\om_1}\dots\lambda_{\om_{n'}}} < \frac{2R}{\lambda_{\min}} \cdot 2^{-qn}.
	\end{equation}
	\item [ii.] By Corollary \ref{cor:entropydim} and the fact that $\PP$ is Bernoulli, one has
	\[
	\PP \left( \left\{ \om : \left| \frac{1}{qn}H(\eta^{(T^{n'}\om)},\D_{qn}) - \alpha\right| < \delta\right\}\right) = \PP {\left( \left\{ \om : \left| \frac{1}{qn}H(\eta^{\pom},\D_{qn}) - \alpha\right| < \delta\right\}\right)}
	\]	
tends to $1$ as $n\to\infty$.
\end{enumerate}
Also, by Proposition \ref{propentro}(v) for every raw component $\nu^{(\om,n')}_{x,i}$ one has
\[
\frac{1}{qn} H( \nu^{(\om,n')}_{x,i} \ast \tau^{(\om,n')},\D_{(q+1)n}) \geq \frac{1}{qn} H( \tau^{(\om,n')}, \D_{(q+1)n}) - O\left(\frac{1}{qn}\right),
\]
Combining this with Equations \eqref{eq:proof2} and \eqref{eq:proof3}, we get for any $\varepsilon,\delta > 0$ that
\[
\lim_{n \rightarrow +\infty}  \PP \left( \left\{ \omega \in \Omega :  \mathbf{P}_n \left( \left|\frac{1}{qn} H( \nu^{(\om,n')}_{x,i} \ast \tau^{(\om,n')},\D_{(q+1)n}) - \alpha\right| < \delta \right) > 1-\varepsilon\right\}\right)=1.
\]
The proof is concluded from a final application of \eqref{eq:proof3}.
\end{proof}

\begin{proof}[Proof of Theorem \ref{teo2}]
Fix $q \in \N$ and, given $\om \in \Omega$ and $\delta > 0$, let us consider (if it exists) a raw component $\nu^{(\om,n')}_{x,n}$ satisfying that
\begin{equation} \label{eq:entro1}
\left|\frac{1}{qn} H( \nu^{(\om,n')}_{x,n} \ast \tau^{(\om,n')},\D_{(q+1)n}) - \frac{1}{qn}H(\tau^{(\om,n')},\D_{(q+1)n})\right| < \frac{\delta}{2}.
\end{equation} Then, by scaling the measures $\nu^{(\om,n')}_{x,n}$ and $\tau^{(\om,n')}$ by $(\lambda_{\om_{1}}\dots \lambda_{\om_{n'}})^{-1}$, using the fact that $\lambda_{\om_{1}}\dots \lambda_{\om_{n'}} \approx 2^{-n}$ a new application of Proposition \ref{propentro}(iv) yields for $n$ sufficiently large (depending only on $\delta$)
\begin{equation} \label{eq:entro2}
H_{qn}( (\beta_n \cdot \nu^{(\om,n'),x,n}) \ast \eta^{(T^{n'} \om)}) < H_{qn}(\eta^{(T^{n'} \om)}) + \delta
\end{equation} where $\beta_n:=(2^n \lambda_{\om_{1}}\dots \lambda_{\om_{n'}})^{-1}$. Indeed, if for any $r > 0$ and countable partition $\mathcal{E}$ of $\R$ we define
$r\cdot \mathcal{E} :=\{ r\cdot E : E \in \mathcal{E}\}$, then we have
\begin{align*}
H( \nu^{(\om,n')}_{x,n} \ast \tau^{(\om,n')},\D_{(q+1)n}) &= H( \nu^{(\om,n')}_{x,n} \ast \tau^{(\om,n')},2^{-n}\cdot \D_{qn})\\
& = H\left( (\lambda_{\om_{1}}\dots \lambda_{\om_{n'}})^{-1} \cdot \left(\nu^{(\om,n')}_{x,n} \ast \tau^{(\om,n')}\right), \beta_n\cdot \D_{qn}\right)\\
& = H( \beta_n \cdot (2^n \cdot \nu^{(\om,n')}_{x,n}) \ast \eta^{(T^{n'}\omega)},\beta_n \cdot \mathcal{D}_{qn})\\
& = H ( (\beta_n \cdot \nu^{(\om,n'),x,n}) \ast \eta^{(T^{n'}\omega)},\beta_n \cdot \mathcal{D}_{qn}) + O(1)\\
& = H ( (\beta_n \cdot \nu^{(\om,n'),x,n}) \ast \eta^{(T^{n'}\omega)}, \mathcal{D}_{qn}) + O_{R,\lambda_{\min}}(1),
\end{align*} where in the second equality we have used that $H(r\cdot \nu, r \cdot \mathcal{E})= H(\nu,\mathcal{E})$ holds for any $r>0$, probability measure $\nu$ and partition $\mathcal{E}$, and in the fourth and fifth inequalities we have used Proposition \ref{propentro}(iv) together with the facts that $\nu^{(\om,n'),x,n}$ is obtained from $2^n \cdot \nu^{(\om,n')}_{x,n}$ by translation (and that the convolution commutes with translations) and that $\beta_n \approx 1$, respectively. From this and \eqref{eq:entro1}, using \eqref{eq:equivtau} it is now easy to obtain \eqref{eq:entro2}.

The idea then is to apply the inverse theorem (Theorem \ref{thm:inverse}) on \eqref{eq:entro2} to conclude. The problem is that $\eta^{(T^{n'} \om)}$ will not always fall under the hypotheses of the theorem, but will rather do so only with high probability. Hence, we must refine our argument.

Consider a sequence $(\varepsilon_j)_{j \in \N} \subseteq \R_{> 0}$ such that $\sum_j \eps_j <+\infty$ and, using Corollary \ref{cor:uedp} and the fact that $\PP$ is Bernoulli, for each $j \in \N$ choose $m_j \in \N$ such that for all $n$ sufficiently large
\begin{equation} \label{eq:uedp}
\PP \left( \left\{ \omega \in \Omega : \mathbf{P}_{\{1,\dots,n\}} \left( \left| H_{m_j}\left(\eta^{(T^{n'}\om),x,i}\right) - \alpha\right| < \varepsilon_j \right) > 1 - \varepsilon_j \right\}\right) > 1 - \frac{\varepsilon_j}{2}.
\end{equation}
Furthermore, for each $j \in \N$ take $\delta_j := \delta(\varepsilon_j,\frac{2R}{\lambda_{\min}},m_j)$ given by Theorem \ref{thm:inverse}, and choose $n_j \geq n(\varepsilon_j,\delta_j,\frac{2R}{\lambda_{\min}},m_j)$, where $n(\varepsilon_j,\delta_j,\frac{2R}{\lambda_{\min}},m_j)$ is also the one from Theorem \ref{thm:inverse}, such that
\begin{equation}  \label{eq:uedp2}
\PP(\mathcal{G}_{\eps_j,\frac{\delta_j}{2},q}^{(n_j)})\ge  1- \frac{\varepsilon_j}{2},
\end{equation}
where $\mathcal{G}_{\eps_j,\delta_j,q}^{(n_j)}$ are the events from Proposition \ref{prop:intermediate-teo2}.

For each $j \in \N$ let $\Omega_j$ be the intersection of the events in \eqref{eq:uedp} and \eqref{eq:uedp2}, where we put $n=n_j$ in \eqref{eq:uedp} (making $n_j$ larger if needed). Assume that $\alpha<1$ and fix $\om\in\Omega_j$. By Theorem \ref{thm:inverse} and since $\omega$ belongs to the event in \eqref{eq:uedp}, we have the implication
\[
H_{qn_j}(\nu) > \varepsilon_j \Longrightarrow H_{qn_j}(\nu* \eta^{(T^{n'}\om)}) \ge H_{qn_j}(\eta^{(T^{n'}\om)}) + \delta_j,
\]
for all $\nu\in \mathcal{P}([0,\frac{2R}{\lambda_{\min}}])$ and all $j$ large enough so as to guarantee that $\alpha+ \varepsilon_j < 1 - \varepsilon_j$. Now, since $\beta_n \cdot \nu^{(\om,n'),x,n} \in \mathcal{P}([0,\frac{2R}{\lambda_{\min}}])$ due to \eqref{eq:compn}, by \eqref{eq:entro2} we obtain that
\[
H_{qn_j}(\beta_n \cdot \nu^{(\om,n'),x,n}) \leq \varepsilon_j
\]
for any raw component $\nu^{(\om,n')}_{x,n}$ satisfying \eqref{eq:entro1} with $\delta=\delta_j$, if $n_j$ is taken sufficiently large (depending only on $\delta_j$). Moreover, since $\beta_n \approx 1$ by \eqref{eq:compn}, using Proposition \eqref{propentro}(iv) again, we see that for $n_j$ sufficiently large (depending only on $\varepsilon_j$) we have in fact that
\[
H_{qn_j}(\nu^{(\om,n')}_{x,n}) \leq 2\varepsilon_j.
\]
Thus, since $\omega$ belongs to the event in \eqref{eq:uedp2}, we get
\[
\mathbf{P}_{n_j} \left( H_{qn_j}( \nu^{(\om,n_j'),x,i}) \leq 2\varepsilon_j \right) > 1 - \varepsilon_j.
\]
Finally, by the identity in \eqref{eq:identity} and the fact that $H_n(\nu)\leq 1$ for any measure $\nu$ supported on the interval $[0,1]$, we deduce that
 \[
 \frac{1}{qn_j} H( \nu^{(\om,n_j')} , \D_{(q+1)n_j}|\D_{n_j}) \leq  3\varepsilon_j.
 \]
 Recall now that this holds for a fixed $\om\in\Omega_j$ where $\PP(\Omega_j)>1-\eps_j$. Since $\sum_j \eps_j<+\infty$, we conclude from the Borel-Cantelli Lemma that
 \[
 \lim_{j \rightarrow +\infty} \frac{1}{qn_j} H( \nu^{(\om,n_j')} , \D_{(q+1)n_j}|\D_{n_j}) = 0
 \]
 for $\PP$-almost every $\omega \in \Omega$. Finally, since the argument above can be repeated exactly to show that for any subsequence $(n_k)_{k \in \N}$ there exists a further subsequence $(n_{k_j})_{j \in \N}$ so that
 \[
 \frac{1}{qn_{k_j}} H(\nu^{(\om,n_{k_j}')} , \D_{(q+1)n_{k_j}}|\D_{n_{k_j}}) \longrightarrow 0
 \]
 for $\PP$-almost every $\omega$, this immediately yields
 \[
 \frac{1}{qn} H( \nu^{(\cdot,n')} , \D_{(q+1)n}|\D_{n}) \overset{\PP}{\longrightarrow} 0
 \] which concludes the proof.
\end{proof}

\subsection{Proof of Theorem \ref{thm:sup-exp-conc}}\label{subsec:endproof}

We show now how to derive Theorem \ref{thm:sup-exp-conc} from  Theorem \ref{teo2}. First, let us fix $\om \in \Omega^*$ with $\Omega^*$ as in Corollary \ref{cor:entropydim} and, for each $n \in \N$, consider again the decomposition
\[
\eta^\pom = \nu^{(\om,n')} \ast \tau^{(\om,n')},
\]
where $n':=\ell^\pom_n$. The following lemma shows that $\nu^{(\om,n')}$ is a good approximation of $\eta^\pom$ at scale $n$.

\begin{lemma}\label{lemma:convesp}
For each $\omega\in\Omega^*$,
\[
\lim_{n \rightarrow +\infty} H_n(\nu^{(\om,n')}) = \alpha.
\]
\end{lemma}
\begin{proof}
By the definitions of $\nu^{(\om,n')}$ and $\eta^\pom$, we have that
\begin{align*}
H_n(\nu^{(\om,n')}) &= \frac1n H(\oeta^\pom, (\Pi_\om^{(n')})^{-1}\mathcal{D}_n),\\
H_n(\eta^\pom) &=\frac1n H(\oeta^\pom, (\Pi_\om)^{-1}\mathcal{D}_n).
\end{align*}
In light of Proposition \ref{propentro}(iv), it is enough to show that each element of $(\Pi_\om^{(n')})^{-1}\mathcal{D}_n$ meets at most $O(1)$ elements of $(\Pi_\om)^{-1}\mathcal{D}_n$ and vice versa. But this is indeed the case, since
\[
\|\Pi_\om-\Pi_\om^{(n)}\|_\infty \le \lambda_{\om_1}\cdots\lambda_{\om_{n'}} (1-\lam_{\max})^{-1}\max_{i,u}{|t_u^{(i)}|} = O(2^{-n}).
\]\end{proof}

In particular, since $\Omega^*$ is a full $\PP$-measure set, the convergence in Lemma \ref{lemma:convesp} above takes place $\PP$-almost surely. Thus, if $\alpha < 1$ then for any fixed $q \in \N$ the decomposition
\[
\frac{H(\nu^{(\om,n')},\mathcal{D}_{(q+1)n})}{n} = H_n(\nu^{(\om,n')}) + \frac{H(\nu^{(\om,n')},\D_{(q+1)n}|\D_n)}{n}
\] together with Theorem \ref{teo2} yields that
\[
\frac{H(\nu^{(\cdot,n'(\cdot))},\mathcal{D}_{(q+1)n})}{n} \overset{\PP}{\longrightarrow} \alpha.
\] In particular, if also $\alpha < \textup{s-dim }(\Sigma)$, then there exists $\delta > 0$ such that
\begin{equation} \label{eq:ac}
\lim_{n \rightarrow +\infty} \PP \left( \left\{ \om \in \Omega :  \frac{H(\nu^{(\om,n')},\D_{(q+1)n})}{n} < \textup{s-dim}(\Sigma) - \delta \right\}\right) = 1.
\end{equation} Now, observe that for the approximation $\nu^{(\om,n')}$ we have the following representation:
	\[
	\nu^{(\om,n')} = \sum_{u \in \X^\pom_{n'}} p^\pom_{u} \cdot \delta_{f^\pom_u(0)}.
	\] Thus, if each value $f^\pom_u(0)$ belonged to a different element of $\D_{(q+1)n}$ then we would have
	\begin{equation}
	\label{eq:ac2}
	\frac{H(\nu^{(\om,n')},\D_{(q+1)n})}{n} = \frac{ \sum_{u \in \X^\pom_{n'}} p^\pom_{u} \log( p^\pom_{u})}{n} = \frac{n'}{n} \cdot  \frac{\sum_{i=1}^{n'} H(p_{\om_i})}{n'}=: \textup{s-dim}(\Sigma^\pom_n).
	\end{equation} Hence, if the left-hand side of \eqref{eq:ac2} is different from $\textup{s-dim}(\Sigma^\pom_n)$, then at least two of these values $f^\pom_u(0)$ must belong to the same element of $\D_{(q+1)n}$, implying that $\Delta^\pom_n \leq 2^{-(q+1)n}$. But applying the ergodic theorem to the maps $\om\mapsto H(p_{\om_1})$ and $\om\mapsto \log\lam_{\om_1}$, yields that $\textup{s-dim}(\Sigma^\pom_n) \rightarrow \textup{s-dim}(\Sigma)$ for $\PP$-a.e. $\omega \in \Omega$. Combining this fact with \eqref{eq:ac}, we conclude
	\[
	\lim_{n \rightarrow +\infty} \PP \left( \left\{ \om \in \Omega :  \frac{\log \Delta^\pom_n}{n} \leq -(q+1)\right\}\right) = 1.
	\] Since $q \in \N$ was arbitrary, the result immediately follows.

\section{Fourier transform estimates}
\label{sec:EK}

Throughout this section we will assume that $I=\{1,\ldots,N\}$ for notational simplicity.

Theorem \ref{thm:Fourier-decay} is a direct corollary of the following more precise statement.
\begin{theorem} \label{thm:Fourier-precise}
let $(\beta_i)_{i\in \{1,\ldots,N\}}$ be strictly positive numbers, and let $\PP$ be a fully supported Bernoulli measure on $\Omega=I^\N$.

Fix a compact set
\[
H= [a,b]\times [\beta_{\min},\beta_{\max}]\times [p_{\min},p_{\max}] \subset (1,\infty) \times (0,\infty)\times (0,1),
\]
and fix also $\alpha>0$.

Then there exist a Borel set $\mathcal{G}_{H,\alpha}\subset \Omega\times [b^{-1},a^{-1}]$ and a  constant $\sig=\sig(H,\alpha)>0$ such that the following hold:
\begin{enumerate}
\item [(i).]
For $\PP$-almost all $\om$,
\[
\dim_H\{\lam\in [b^{-1},a^{-1}]: (\om,\lam)\notin \mathcal{G}_{H,\alpha} \} \le \alpha.
\]
\item[(ii).] If $(\om,\lam)\in\mathcal{G}_{H,\alpha}$, and $\Sigma=((\Phi^{(i)})_{i \in I},(p_i)_{i \in I},\PP)$ is a non-degenerate model such that:
\begin{itemize}
 \item for each $i \in I$, the mappings $f_j^{(i)}$ have contraction ratio $\lam^{\beta_i}$,
 \item $p_j^{(i)}\in [p_{\min},p_{\max}]$ for all $i \in I$ such that $k_i\ge 2$ and all $j\in\{1,\ldots,k_i\}$,
\end{itemize}
then $\eta^\pom\in\mathcal{D}_1(\sigma)$.
\end{enumerate}

\end{theorem}

Theorem \ref{thm:Fourier-decay} follows immediately by taking
\[
\mathcal{G} = \bigcup_{k,n} \mathcal{G}_{H_n,1/k},
\]
where $H_n = [1+1/n,n]\times [\min \beta_i,\max\beta_i] \times [1/n,1-1/n]$.

From now on, we fix the numbers $\beta_i$ and the Bernoulli measure $\PP$, as in the statement of Theorem \ref{thm:Fourier-precise}.

\subsection{Probabilistic estimates}

Let
\[
X_1 = X_1(\om) = \min\{i\ge 1:\ \om_i=1\},\ \ \ X_{n+1}(\om) = X_1(T^{X_1(\om)+\dots +X_n(\om)}\om),\ \ n\ge 1.
\]
Thus, $X_{n+1}(\om)$ is the waiting time between the $n$-th and $(n+1)$-st appearance of 1 in the sequence $\om$. (Later on we will assume that the symbol $1$ ensures non-degeneracy of the model, i.e.  $k_1\ge 2$ and $t_1^{(1)} \ne t_2^{(1)}$.) The process $(X_n)_{n\in \N}$ is defined almost surely (when there are infinitely many 1's in $\om$). Since $\PP$ is Bernoulli, $(X_n)_{n\in \N}$ is an i.i.d.\ sequence of exponential random variables.

Note that there exists $\eps>0$ such that
\be \label{expo}
C_1:= \E[e^{\eps X_1}]\in (1,\infty).
\ee

Before embarking on the proof of Theorem~\ref{thm:Fourier-precise}, we state the property which provides the full measure set appearing in the first part of Theorem~\ref{thm:Fourier-precise}.

\begin{lemma} \label{lem:proba}
Consider the process $(X_n)_{n\in \N}$ defined as above. There exists a positive constant $L_1$ such that for $\PP$-a.e. $\om$, for any $\varrho>0$ and all $M$ sufficiently large (depending on $\varrho$ and $\om$),
\be \label{cond2}
\max\left\{\sum_{n\in \Psi} X_n(\om):\ \Psi \subset \{0,\ldots,M-1\}, \ |\Psi| \le \varrho M\right\} \le L_1\cdot \log(1/\varrho)\cdot \varrho M.
\ee
\end{lemma}

\begin{remark}
Observe that for any given $\wtil{\varrho}>0$, we have for $\PP$-almost all $\om$ and all $n$ sufficiently large (depending on $\om$) that
\be \label{cond1}
X_n(\om) \le \wtil{\varrho} n;
\ee
Indeed, $\frac{1}{n} \sum_{i=1}^n X_i \to \E[X_1]< \infty$ almost surely, by the Law of Large Numbers, and hence $X_n(\om)/n\to 0$ for $\PP$-a.e.\ $\om$.
\end{remark}

\begin{proof}[Proof of Lemma~\ref{lem:proba}]
Note that, since the $(X_i)_{i \in \N}$ are i.i.d., by the exponential Tchebychev inequality we have
\begin{equation}
\label{fact}
\PP\Bigl[\sum_{i=1}^{n} X_i \ge Kn\Bigr] \le  \E[e^{\varepsilon \sum_{i=1}^n X_i}]e^{-\varepsilon K n} = C_1^n e^{-\varepsilon Kn} \leq e^{-\eps Kn/2}
\end{equation}
for any  $K\ge 2\log C_1/\eps$, where $\eps$ and $C_1$ are the ones from (\ref{expo}). Furthermore, again using that $(X_i)_{i \in \N}$ are i.i.d., we can use the version of \eqref{fact} in which $\sum_{i=1}^n X_i$ is replaced by $\sum_{k=1}^n X_{i_k}$ for an arbitrary (but deterministic) increasing sequence $i_1 < \ldots < i_n$.
Now, consider the event
\[
\Wk(M,\varrho,K) = \left\{\max_{\stackrel{\scriptstyle{\Psi\subset \{0,\ldots,M-1\}}}{|\Psi|\le\varrho M}} \sum_{n\in \Psi} X_n \ge K(\varrho M)\right\}
\]
Then we have for $K \ge 2\log C_1/\eps$,
\begin{eqnarray*}
\PP\bigl(\Wk(M,\varrho,K)\bigr)
& \le & \!\!\!\!\sum_{\stackrel{\scriptstyle{\Psi\subset \{0,\ldots,M-1\}}}{|\Psi|\le\varrho M}}  \PP\left[ \sum_{n\in \Psi} X_n \ge K(\varrho M)\right] \\[1.2ex]
& \le & \sum_{i\le \varrho M} {M\choose i}e^{-\eps K(\varrho M)/2},
\end{eqnarray*}
in view of (\ref{fact}).
By Stirling's formula,  there exists $C_2>1$ such that
\[\sum_{i\le \varrho M} {M\choose i}\le \exp\left[C_2\varrho \log(1/\varrho)M\right]\ \ \mbox{for}\ \  \varrho<e^{-1}\ \ \mbox{and all}\ \  M>1.
\]
Therefore,
\[
 \sum_{i\le \varrho M} {M \choose i} e^{-\eps K(\varrho M)/2} \le \exp[-\eps K (\varrho M)/4] \ \ \mbox{for}\ \ K \ge \frac{4C_2}{\eps} \log(1/\varrho).
\]
Now the Borel-Cantelli Lemma implies that the event $\Wk(M,\varrho, L_1 \log(1/\varrho))$ does not occur for all $M$ sufficiently large, $\PP$-almost surely, with
\[
L_1:= \eps^{-1} \max\{4C_2,2\log C_1\},
\]
and this is precisely the claim of Lemma~\ref{lem:proba}.
\end{proof}

\subsection{Proof of Theorem~\ref{thm:Fourier-precise}}

We are now ready to start the proof of Theorem~\ref{thm:Fourier-precise}.

\begin{proof}[Proof of Theorem~\ref{thm:Fourier-precise}] Let $\Sigma$ be a non-degenerate model as in the second part of the statement.                                                                                                                                                                                                                                                                                                                                                                                                                                                                                                                                                                                                                                                                                                                                                                                                                                                                                                                                                                                                                                                                                                                                                                                                                                                                                                                                                                                                                                                                                                                                                                                                                                                                                                                                                                                                                                                                                                                                                                                                                                                                                                                                                                                                                                                                                                                                                                                                                                                                                                                                                                                                                                                                                                                                                                                                                                                                                                                                                                                                                                                                                                                                                                                                                                                                                                                                                                                                                                                                                                                                                                                                                                                                                                                                                                                                                                                                                                                                                                                                                                                                                                                                                                                                                                                                                                                                                                                                                                                                                                                                                                                                                                                                                                                                                                                                                                                                                                                                                                                                                                                                                                                                                                                                                                                                                                                                                                                                                                                                                                                                                                                                                                                                                                                                    Without loss of generality, we may assume that $k_1\ge 2$, $t_1^{(1)}\neq t_2^{(1)}$. Since translating and scaling a measure does not change whether the measure is in $\mathcal{D}_1(\sigma)$, we may assume also that $t_1^{(1)}=0$, and $t_2^{(1)}=1$. Also, the statement of the proposition remains equivalent if we simultaneously replace $\beta_i$ by $c\beta_i$, and $[a,b]$ by $[a^{1/c}, b^{1/c}]$, for some $c>0$. Thus, we can assume without loss of generality that $\beta_1=1$. We may also assume that $N\ge 2$, otherwise this is a standard homogeneous self-similar measure and the claim follows e.g. from \cite[Proposition 2.3]{Shmerkin14}.

We write $\eta^\pom_\lam$ for the measures generated by the model $\Sigma$; we keep in mind that there is also a dependence on the probabilities and the translations. We write $\lam_i$ for the contraction ratio of the IFS corresponding to the symbol $i$, and recall that  $\lam_i = \lam^{\beta_i}$ by assumption. By definition, $\eta^{(\om)}_\lam$ is the distribution of a sum of independent random variables, hence it can be represented as an infinite convolution measure
\[
\eta^{(\om)}_\lam = \Conv_{n\in\N} \left(\sum_{j=1}^{k_{\om_n}} p_j ^{(\om_n)} \delta_{\lam_{\om_1}\cdots \lam_{\om_{n-1}} t_j^{(\om_n)}}\right),
\]
where $\delta_x$ is Dirac's delta centered at $x$.
We thus have
\begin{eqnarray}
|\widehat{\eta_\lam^{(\om)}}(\xi)| & = & \prod_{n\in \N} \left| \sum_{j=1}^{k_{\om_n}} p_j ^{(\om_n)} e^{\pi i \lam_{\om_1}\cdots \lam_{\om_{n-1}} t_j^{(\om_n)}\xi}\right| \nonumber \\
                                         & \le & \prod_{n:\,\om_n=1} \left( \left| p_1^{(1)} + p_2^{(1)} e^{\pi i \lam_{\om_1}\cdots \lam_{\om_{n-1}}  \xi}\right| + \Bigl(1 - p_1^{(1)}-p_2^{(1)}\Bigr)\right) \nonumber \\
                                         & \le & \prod_{n:\,\om_n=1} \left( 1 - c_0 \|\lam_{\om_1}\cdots \lam_{\om_{n-1}} \xi\|^2\right), \label{ineq1}
\end{eqnarray}
where $\|x\|$ denotes the distance from $x$ to the nearest integer and $c_0>0$ is a constant depending only on $[p_{\min},p_{\max}]$.

We will impose a number of conditions on $\om$ which hold $\PP$-a.e. First of all, clearly a.e.\ $\om$ contains infinitely many 1's.
Let \[
\om = W_1 W_2 W_3\ldots,\ \ \mbox{with}\ \ W_i=W'_i1,
\]
where $W'_i$ are words in the alphabet $\{1,\ldots,N\}$ not containing $1$'s (the $W'_i$ may be empty).

For a finite word $v=v_1\ldots v_n$ let $\lam_v = \lam_{v_1\ldots v_n}:= \lam_{v_1}\cdots \lam_{v_n}$. Then (\ref{ineq1}) can be rewritten as
\be \label{ineq2}
|\widehat{\eta_\lam^{(\om)}}(\xi)| \le \prod_{k=1}^\infty \left(1- c_0 \|\lam_{W_1\ldots W_{k-1} W'_k}\cdot \xi\|^2\right),
\ee
with the convention that the word equals $W'_1$ when $k=1$.

Denote $\theta=\lam^{-1}=\lam^{-\beta_1}$, $\theta_i = \lam_i^{-1}=\lam^{-\beta_i}$ and $\theta_v = \lam_v^{-1}$, and let
\[
\Theta^{(k)}:= (\lam_{W_1\ldots W_{k-1} W'_k})^{-1} = \theta_{W_1}\cdots \theta_{W_{k-1}} \theta_{W'_k}.
\]
Now suppose $\xi \in [\Theta^{(M)}, \Theta^{(M+1)}]$ for some $M\in \N$ (this depends on the $\lam_i$, and we will keep this in mind), and let $\tau = \xi/\Theta^{(M)}$.
Let
\[
\Theta^{(M)}_{k}:= \frac{\Theta^{(M)}}{\Theta^{(M-k)}},\ \ k=0,\ldots,M-1;
\]
then (\ref{ineq2}) yields
\be \label{ineq3}
|\widehat{\eta_\lam^{(\om)}}(\xi)| =  |\widehat{\eta_\lam^{(\om)}}(\Theta^{(M)}\cdot \tau)| \le \prod_{k=0}^{M-1} \bigl( 1 - c_0 \|\Theta_k^{(M)} \tau\|^2\bigr).
\ee
Note that $\tau \in [1, \Theta_1^{(M+1)}]$.

The following proposition captures the key combinatorial estimate at the heart of  Theorem~\ref{thm:Fourier-precise}. It is a typical ``Erd\H{o}s-Kahane'' estimate, although we need to take extra care in our setting due to the fact that we deal not with a deterministic parametrized family, but with a random collection of parametrized families. We should note that the rest of the proof has many common features with another version of the Erd\H{o}s-Kahane argument in a random setting, which appeared in \cite[Section 10]{BufetovSolomyak15}.

\begin{proposition} \label{prop:Four} Fix a compact set $H$ as in the statement of Theorem~\ref{thm:Fourier-precise} and $\alpha>0$.
There exist positive constants $\rho = \rho(H,\alpha)$ and $\delta=\delta(H,\alpha)$  such that, for $\PP$-a.e.\ $\om$, there exists $M_0=M_0(\om,H)$ such that for any $M\ge M_0$ and  any $\beta_1,\ldots,\beta_N \in [\beta_{\min},\beta_{\max}]$, the set
\be \label{badset}
E_{\rho,\delta,M}:=\left\{\theta \in [a,b]:\ \max_{\tau \in [1, \Theta_1^{(M+1)}]} \frac{1}{M} \bigl| \{k\in \{0,\ldots,M-1\}:\,\|\Theta_k^{(M)} \tau\| \ge \rho\}\bigr| < \delta\right\}
\ee
can be covered by $O_H(a^{\alpha M/2})$ balls of radius $O_H(a^{-M/2})$. Note that $E_{\rho,\delta,M}$ is a random set, depending on $\om$.
\end{proposition}

First we derive Theorem~\ref{thm:Fourier-precise} from Proposition \ref{prop:Four}. Let $\Omega_1$ be the full measure set from the proposition, and let
\[
\mathcal{G}_{H,\alpha} = \left\{ (\om,\lam):\lam^{-1}\notin   \limsup_{M \rightarrow +\infty} E_{\rho,\delta,M} \right\}.
\]
It is easy to see that $\mathcal{G}_{H,\alpha}$ is Borel: the numbers $\Theta_k^{(M)}$ depend on $\om$ and $\lam$ in a Borel manner, and because $\|\cdot\|$ is continuous, the set \eqref{badset} can be defined in terms of maxima over rational $\tau$, rather than all $\tau$.

Now fix $\om\in\Omega_1$. Define the exceptional set $E=E(\om)= \{ \lam: (\om,\lam)\notin \mathcal{G}_{H,\alpha}\}$. Explicitly,
\[
E^{-1} :=\{\lambda^{-1}: \lambda \in E\} = \limsup_{M \rightarrow +\infty} E_{\rho,\delta,M}.
\]
By Proposition \ref{prop:Four} we have  $\dim_H(E)=\dim_H(E^{-1}) \le \alpha$.

On the other hand, for $\lam\not\in E$ we have that $\lam^{-1}\not\in  E_{\rho,\delta,M}$ for all $M\ge M_1=M_1(\om,\lam)$.
Then for $\xi\ge \Theta^{(M_1)}$ choose $M\ge M_1$ such that $\xi \in [\Theta^{(M)}, \Theta^{(M+1)}]$ and obtain from (\ref{ineq3}) that
\be \label{ineq4}
|\widehat{\eta_\lam^{(\om)}}(\xi)| \le (1-c_0 \rho^2)^{\delta M} .
\ee
By the Law of Large Numbers there exists $c_1>0$ such that for $\PP$-a.e.\ $\om$,
\be \label{Birk}
\left|\{n\in\{1,\ldots,L\}:\,\om_n = 1\}\right| \ge c_1L
\ee
for $L$ sufficiently large (in fact, we can take any $c_1< \PP(\om_1=1)$). Hence, for $M$ sufficiently large we have
\[
|W_1\ldots W_{M+1}| \le c_1^{-1}(M+1).
\]
Then $\xi \le \Theta^{(M+1)} \le b^{\beta_{\max}c_1^{-1}(M+1)}$, where $\beta_{\max}=\max\{\beta_i:i\in I\}$, and (\ref{ineq4}) yields
\[
|\widehat{\eta_\lam^{(\om)}}(\xi)|  = O_\om(|\xi|^{-\sigma}),\ \ \mbox{where}\ \ \sigma = -\delta \beta_{\max}^{-1} \cdot c_1 \log(1-c_0 \rho^2)/\log b,
\]
as desired. Note that $\sig$ depends on $H$, but not on $\om$ --- just on the model.
\end{proof}

\begin{proof}[Proof of Proposition~\ref{prop:Four}]
Fix $M\in\N$. For $n=0, \ldots, M-1$, let
\be \label{eqa1}
\Theta^{(M)}_n \tau = K_n^{(M)} + \eps_n^{(M)},\ \ \ \mbox{where}\ \ K_n^{(M)}\in \N,\ \ |\eps_n^{(M)}|\le 1/2,
\ee
so that $\|\Theta^{(M)}_n \tau \|=\eps_n^{(M)}$. Note that $\tau\ge 1$ and $\Theta^{(M)}_n\ge \theta^n \ge a^n$, so
\be \label{Kbound}
K_n^{(M)}\ge \max\{1, a^n-1\}.
\ee
We will next drop the superscript $(M)$ from the notation of $K_n$ and $\eps_n$, but will keep the dependence in mind. These numbers also depend on $\lam_i$ and $\tau$, of course.
All the constants implicit in the $O(\cdot)$ notation below are allowed to depend on $H$.

It follows from (\ref{eqa1}) that
\be \label{eqa2}
\frac{\Theta_{n+1}^{(M)}}{\Theta_n^{(M)}} - \frac{K_{n +1}}{K_n} = \frac{\eps_{n+1}}{K_n} - \frac{\Theta_{n+1}^{(M)}}{\Theta_n^{(M)}}\frac{\eps_n}{K_n}\,.
\ee
Observe that
\[
\frac{\Theta_{n+1}^{(M)}}{\Theta_n^{(M)}}= \theta_{W_{M-n}} = \theta^{\beta(W_{M-n})},\ \ \ \mbox{where}\ \ \beta(v_1\ldots v_\ell):= \sum_{j=1}^{\ell} \beta_{v\_j}.
\]
We obtain from (\ref{eqa2}):
\be \label{equ31}
\left|\theta^{\beta(W_{M-n})} - \frac{K_{n+1}}{K_n}\right| \le\frac{\theta^{\beta(W_{M-n})} |\eps_n| + |\eps_{n+1}|}{K_n}\,.
\ee
Replacing $n$ by $n+1$ gives
\[
\left|\theta^{\beta(W_{M-(n+1)})} - \frac{K_{n+2}}{K_{n+1}}\right| \le\frac{\theta^{\beta(W_{M-(n+1)})} |\eps_{n+1}| + |\eps_{n+2}|}{K_{n+1}}\,.
\]
It is important that $\beta(W_j)$ does not depend on $\theta$, but only on the sequence $\om$.
Using the inequality
\[
|x^s-y^s| \le |x-y|,\ \  \mbox{for}\ x,y>0,\ s\in (0,1],
\]
we obtain from (\ref{equ31}):
\be \label{equ4}
\left|\theta - \Bigl(\frac{K_{n+1}}{K_n}\Bigr)^{\beta(W_{M-n})^{-1}}\right| \le\frac{\theta^{\beta(W_{M-n})} |\eps_n| + |\eps_{n+1}|}{K_n}\,,
\ee
since $W_{M-n}$ contains a 1 and $\beta_1 = 1$ by assumption.
Similarly,
\be \label{equ41}
\left|\theta- \Bigl(\frac{K_{n+2}}{K_{n+1}}\Bigr)^{\beta(W_{M-(n+1)})^{-1}} \right| \le\frac{\theta^{\beta(W_{M-(n+1)})} |\eps_{n+1}| + |\eps_{n+2}|}{K_{n+1}}\,.
\ee
From (\ref{equ31}), using trivial bounds, we obtain
\be \label{eq-triv1}
\frac{K_{j+1}}{K_{j}} \le \theta^{\beta(W_{M-j})} +  \frac{\theta^{\beta(W_{M-j})}+1}{2K_{j}}\le 2 \theta^{\beta(W_{M-j})} \le (2b)^{\beta(W_{M-j})},
\ee
for $j=n,n+1$. Combining  (\ref{equ4}) and (\ref{equ41}) and using the inequality
\[
|x^s - y^s| \le s \max(x^{s-1}, y^{s-1})\cdot |x-y|, \ \ \mbox{for}\ x,y>0,\ s \ge 1,
\]
 with $s = \beta(W_{M-n-1})$, yields, after a little calculation using (\ref{eq-triv1}),
 \begin{eqnarray} \label{equ5}
& & \left| \frac{K_{n+2}}{K_{n+1}} - \Bigl(\frac{K_{n+1}}{K_{n}}\Bigr)^{\frac{\beta(W_{M-n-1})}{\beta(W_{M-n})}}\right|  \\[1.1ex]
& \le & L_n \left[ \frac{\theta^{\beta(W_{M-n})} |\eps_n| + |\eps_{n+1}|}{K_n}+ \frac{\theta^{\beta(W_{M-n-1})} |\eps_{n+1}| + |\eps_{n+2}|}{K_{n+1}}\right], \nonumber
\end{eqnarray}
where
\begin{align*}
L_n &= \beta(W_{M-n-1}) (2 b)^{\beta(W'_{M-n-1})} \\
&= (1+\beta(W'_{M-n-1}))(2 b)^{\beta(W'_{M-n-1})} \le (2eb)^{\beta(W'_{M-n-1})}.
\end{align*}
Now, (\ref{equ5}) and then (\ref{eq-triv1})  imply
\begin{eqnarray}
& & \left| K_{n+2}-K_{n+1} \Bigl(\frac{K_{n+1}}{K_{n}}\Bigr)^{\frac{\beta(W_{M-n-1)})}{\beta(W_{M-n})}}\right| \nonumber \\
& \le & L_n\left[ \frac{K_{n+1}}{K_n} \left(\theta^{\beta(W_{M-n})} |\eps_n| + |\eps_{n+1}|\right) + \theta^{\beta(W_{M-n-1})} |\eps_{n+1}| + |\eps_{n+2}|\right] \nonumber \\[1.1ex]
& \le & L_n\left[ (2b)^{\beta(W_{M-n})}(b^{\beta(W_{M-n})}+1) + (b^{\beta(W_{M-n-1})}+1)\right] \cdot \max(|\eps_{n}|,|\eps_{n+1}|,|\eps_{n+2}|) \nonumber \\[1.1ex]
& \le &  2(2eb)^{\beta(W'_{M-n-1})} \left[(2b^2)^{\beta(W_{M-n})} + b^{\beta(W_{M-n-1})}\right] \cdot \max(|\eps_{n}|,|\eps_{n+1}|,|\eps_{n+2}|) \nonumber \\[1.1ex]
&\label{equ7} \le & 4(12b^3)^{\beta_{\max}(|W_{M-n}| + |W_{M-n-1}|)} \cdot \max(|\eps_{n}|,|\eps_{n+1}|,|\eps_{n+2}|),
\end{eqnarray}
using crude bounds in the last two steps.
Let
\be \label{def-Brho}
B_n:= 4(12b^3)^{\beta_{\max}(|W_{M-n}| + |W_{M-n-1}|)}, \ \ \rho_n = (2B_n)^{-1}.
\ee
Note that these numbers depend on $\om$ and $M$, and recall that the $K_n$ depend, additionally, on $\theta$ and $\tau$.

\begin{lemma} \label{lem-step}
Fix $M\in\N$ and $\om$ such that $\om$ has infinitely many $1$s. Then

{\rm (i)} given $K_n, K_{n+1}$, there are at most $2B_n+1$ possibilities for $K_{n+2}$, independent of $\theta\in [a,b]$ and $\tau\ge 1$;

{\rm (ii)} if $\max(|\eps_{n+2}|,|\eps_{n+1}|,|\eps_{n}|)<\rho_n$, then $K_{n+2}$ is uniquely determined by $K_n$ and $K_{n+1}$, independent of $\theta\in [a,b]$ and $\tau\ge 1$.
\end{lemma}
The claims of the lemma are immediate from (\ref{equ7}), using the fact that $K_{n+2}$ is an integer.

Let $\wtil{E}_{\delta,M}$ be defined by
\begin{eqnarray*}
\wtil{E}_{\delta,M} & := & \Bigl\{\theta\in [a,b]:\ \mbox{there exists}\ \tau \in [1, \Theta_1^{(M+1)}]\ \ \mbox{for which}\\
                              &     &  \bigl|\{n\in \{0,\ldots,M-3\}:\ \max\{\|\Theta_n^{(M)}\tau\|, \|\Theta_{n+1}^{(M)}\tau\|, \|\Theta_{n+2}^{(M)}\tau\| \}\ge \rho_n\}\bigr| < \delta M\Bigr\},
\end{eqnarray*}
where $\rho_n$ is the one from (\ref{def-Brho}).
First we claim that $\PP$-almost surely,
\be \label{inclu}
\wtil{E}_{4\delta,M} \supset E_{\rho,\delta,M}
\ee
for $M$ sufficiently large (depending on $\om$ only), with
\[
\rho = (1/8)\cdot (12b^3)^{-4\beta_{\max}/(c_1\delta)},
\]
where $c_1>0$ is the constant from (\ref{Birk}). Recall that
\begin{eqnarray*}
{E}_{\rho, \delta,M} & = & \Bigl\{\theta\in [a,b]:\ \mbox{there exists}\ \tau \in [1, \Theta_1^{(M+1)}]\ \ \mbox{for which}\\
                              &     &  \bigl|\{n\in \{0,\ldots,M-1\}:\ \|\Theta_n^{(M)}\tau\| \ge \rho\}\bigr| < \delta M\Bigr\}.
\end{eqnarray*}
Observe that there are fewer than $\delta M/2$ integers $n\in \{1,\ldots,M-1\}$ for which $|W_{M-n}| >2/(c_1\delta)$, for $M$ sufficiently large, by (\ref{Birk}), hence
there are fewer than $\delta M$ integers $n\in \{1,\ldots,M-1\}$ for which $|W_{M-n}| + |W_{M-n-1}| >4/(c_1\delta)$, for $M$ sufficiently large, hence fewer than $\delta M$ integers $n\in \{1,\ldots,M-1\}$ for which
$\rho_n \le \rho$, in view of (\ref{def-Brho}).
Suppose that $\theta\not\in \wtil{E}_{4\delta,M}$. Then
for all $\tau \in [1, \Theta_1^{(M+1)}]$,
\[
\bigl|\{n\in \{0,\ldots,M-3\}:\ \max\{\|\Theta_n^{(M)}\tau\|, \|\Theta_{n+1}^{(M)}\tau\|, \|\Theta_{n+2}^{(M)}\tau\| \} \ge \rho\}\bigr| \ge 3\delta M,
\]
and therefore,
\[
\bigl|\{n\in \{0,\ldots,M-1\}:\ \|\Theta_n^{(M)}\tau\| \ge \rho\}\bigr| \ge \delta M,
\]
showing that $\theta\not\in E_{\rho,\delta,M}$ and
confirming the claim (\ref{inclu}).

It remains to estimate the number of balls of radius $O_H(a^{-M/2})$ needed to cover $\wtil{E}_{4\delta,M}$. Suppose that $\theta\in \wtil{E}_{4\delta,M}$; choose appropriate $\tau$ from the definition of $\wtil{E}$, and
find the corresponding $K_n,\eps_n$ from (\ref{eqa1}). We have from (\ref{Kbound}) and (\ref{equ4}) that for $n=1,\ldots,M-2$,
\be \label{last1}
\left|\theta - \Bigl(\frac{K_{n+1}}{K_n}\Bigr)^{\beta(W_{M-n})^{-1}}\right| \le\frac{\theta^{\beta(W_{M-n})} |\eps_n| + |\eps_{n+1}|}{K_n} \le \frac{b^{\beta_{\max}|W_{M-n}|}+1}{a^n-1},
\ee
Again using (\ref{Birk}), we note that for $M$ sufficiently large
\[
\min\{|W_{M-n}|:\ n\in [M/2-2, M-2] \}\le 2c_1^{-1},
\]
hence we can find $n\in [M/2-2, M-2]$ for which the right-hand side of (\ref{last1}) is $O_H(a^{-M/2})$. Thus (\ref{last1}) provides an interval of size $O_H(a^{-M/2})$ to cover $\theta$. Since the center of this interval is determined by $K_n$ and $K_{n+1}$ (recall that $W'_{M-n}$ are fixed in advance by $\om$), it suffices to estimate the number of sequences $K_0,\ldots,K_{M-1}$ which may arise this way. Let $\Psi_M$ be the set of $n\in \{0,\ldots,M-3\}$ where we have $\max\{|\eps_n|, |\eps_{n+1}|, |\eps_{n+2}|\}\ge \rho_n$. By the definition of $\wtil{E}_{4\delta,M}$ we have $|\Psi_M|\le 4\delta M$. By Lemma~\ref{lem-step}, for a fixed $\Psi_M$ the number of possible
sequences $K_0,\ldots,K_{M-1}$ is at most
\[
\Bk_M:= \prod_{n\in \Psi_M} (2B_n+1),
\]
times the number of ``beginnings'' $K_0, K_1$. The number of choices for $K_0, K_1$ is bounded above by
\[
\Theta_1^{(M+1)}\cdot (\Theta_1^{(M+1)}\Theta_1^{(M)}) \le b^{\beta_{\max}(2|W_{M+1}| + |W_M|)},
\]
by (\ref{eqa1}), in view of the assumption $\tau \in [1,\Theta_1^{(M+1)}]$.
Let $\wtil{\varrho}=\wtil{\varrho}(H,\alpha)>0$ be such that
\[
b^{3\wtil{\varrho}\beta_{\max}} < a^{\alpha/4}.
\]
Then for $\PP$-a.e.\ $\om$, for $M$ sufficiently large, the number of choices for $K_0, K_1$ is bounded above by $a^{\alpha M/4}$, in view of (\ref{cond1}). It remains to estimate $\Bk_M$. In view of (\ref{cond2}) in Lemma \ref{lem:proba} and (\ref{def-Brho}),
\[
\Bk_M \le \exp \Bigl(O_H(1)\cdot  \sum_{n\in \Psi_M} (|W_{M-n}| + |W_{M-n-1}|)\Bigr) \le \exp\bigl(O_H(1)\cdot L_1\cdot \log(1/4\delta)\cdot \delta M\bigr),
\]
for $M$ sufficiently large. Combining everything, we obtain that the number of balls of radius $O_H(a^{-M/2})$ needed to cover $\wtil{E}_{4\delta,M}$ is not greater than
\[
a^{\alpha M/4}\cdot \sum_{i\le 4\delta M} {M\choose i}\cdot \exp\bigl(O_H(1)\cdot L_1\cdot \log(1/4\delta)\cdot \delta M\bigr),
\]
which can be made smaller than $a^{\alpha M/2}$ by taking appropriate $\delta=\delta(H,\alpha)>0$. This concludes the proof of Proposition~\ref{prop:Four}, and now Theorem~\ref{thm:Fourier-precise} is proved completely.
\end{proof}

\section{Absolute continuity of self-similar measures}
\label{sec:abs-cont}

In this section we prove Theorem~\ref{thm:abs-cont-non-hom}. Fix translations $t_1<\ldots<t_k$ and $\mathbf{p}$ and let
\[
\mathcal{G}= \{ \lam\in (0,1)^k: \nu_{\lam,t}^{\mathbf{p}} \text{ is absolutely continuous}\}.
\]
Recall that for any $\lam,\mathbf{p}$, the measure $\nu_{\lam,t}^\mathbf{p}$, being self-similar, is either absolutely continuous or purely singular.  It follows from \cite[Theorem 1.10]{SimonVago17} that $\mathcal{G}$ is an $F_\sigma$ set. 

Fix positive numbers $1=\beta_1,\ldots,\beta_k$, and write $\nu_\lam^{\mathbf{p}}=\nu_{(\lam^{\beta_1},\ldots,\lam^{\beta_k}),t}^{\mathbf{p}}$ and (with a slight abuse of notation) also $s(\lam,\mathbf{p})=s((\lam^{\beta_1},\ldots,\lam^{\beta_k}),\mathbf{p})$.  By the coarea formula (see e.g. \cite[Theorem 3.10]{EvansGariepy15}), applied to the map
\[
(\lam_1,\ldots,\lam_k)\mapsto (\log(\lam_2)/\log(\lam_1),\ldots,\log(\lam_k)/\log(\lam_1))
\]
and the measurable function $\mathbf{1}_\mathcal{G}$, in order to establish Theorem~\ref{thm:abs-cont-non-hom} it is enough to prove the following:

\begin{proposition} \label{prop:non-hom}
Given $\eps>0$, there exists a null set $E_{\mathbf{p}} \subset (0,1)$ (depending on the $t_i$, $\beta_i$ and $\mathbf{p}$) such that if $\lam\in (0,1)\setminus E_{\mathbf{p}}$ is such that $s(\lam,\mathbf{p})>1+\eps$, then $\nu_\lam^{\mathbf{p}}$ is absolutely continuous.
\end{proposition}

The rest of this section is devoted to the proof of Proposition \ref{prop:non-hom}. We write $( g_{\lam,i}(x)= \lam^{\beta_i}x+t_i)$ for the IFS corresponding to the parameter $\lam$. Moreover, if $u\in\{1,\ldots,k\}^n$, we will write $g_{\lam,u} = g_{\lam,u_1}\circ\cdots\circ g_{\lam,u_n}$ for short.

\begin{lemma} \label{lem:non-hom-to-random-hom}
For any $\delta>0$ there exist $r\in\N$, depending on $\delta$ and $\mathbf{p}$, and a parametrized family of models $\Sigma_\lam=((\Phi^{(i)}_\lam)_{i \in I},(\wtil{p}_i)_{i \in I},\PP)$, such that the following holds:
\begin{enumerate}
\item[(i).] The measure $\PP$ depends only on $\mathbf{p}$ and $r$, the number of maps $k_i$ in $\Phi^{(i)}_\lam$ depends only on $r$ (and $k$), the probability vectors $\wtil{p}_i$ are all uniform, and  if we write
\[\Phi^{(i)}_\lam = (f_{\lam,1}^{(i)},\ldots,f_{\lam,k_i}^{(i)}),\]
then the maps $f_{\lam,j}^{(i)}$ depend on $\lam$ (and $t_i,\beta_i$).
\item[(ii).] $\{ g_{\lam,u}: u\in\{1,\ldots,k\}^r\}$ is the disjoint union of the sets $\{ f_{\lam,1}^{(i)},\ldots,f_{\lam,k_i}^{(i)}\}$, $i\in I$.
\item[(iii).] $\PP$ is a globally supported Bernoulli measure.
\item[(iv).]  $\nu_\lam^{\mathbf{p}} = \int \eta_\lam^\pom\,d\PP(\om)$, where $\eta_\lam^\pom$ are the random measures generated by the model $\Sigma_\lam$.
\item[(v).] $\textup{s-dim}(\Sigma_\lam) > (1-\delta) s(\lam,\mathbf{p})$.
\end{enumerate}
\end{lemma}
\begin{proof}
The proof is very similar to \cite[Lemma 6.6]{GSSY16} but we provide all details for completeness. See also \S\ref{subsec:outline-of-proof} for an alternative explanation of Part (iv) in the special case $k=2$. Fix $r\in \N$, and let
\[
I:= \left\{ (n_1,\ldots,n_k)\in\N_0^k: \sum_{j=1}^k n_j = r\right\}.
\]
For each word $u\in \{1,\ldots,k\}^r$, let $N_j(u)$ be the number of times the symbol $j$ appears in $u$. Let us define a map $\Psi:\{1,\ldots,k\}^r\to I$ as
\[
\Psi(u) = (N_1(u),\ldots,N_k(u)).
\]
Further, for each $\vec{n}\in I$, let
\[
q_{\vec{n}} = \sum_{u\in\Psi^{-1}(\vec{n})} p_{u_1}\cdots p_{u_r}= |\Psi^{-1}(\vec{n})| p_1^{\vec{n}_1}\cdots p_k^{\vec{n}_k}.
\]
Write $k_{\vec{n}}=|\Psi^{-1}(\vec{n})|$, and enumerate $\Psi^{-1}(\vec{n})=(u^{(\vec{n})}_1,\ldots,u^{(\vec{n})}_{k_{\vec{n}}})$. Set also
\begin{align*}
f_{\lam,j}^{(\vec{n})} &= g_{\lam,u_j^{(\vec{n})}},\\
\wtil{p}_j^{(\vec{n})} &= p_{u_j^{(\vec{n})}}/q_{\vec{n}} = k_{\vec{n}}^{-1}.
\end{align*}
Finally, define $\PP$ as the Bernoulli measure on $I^\N$ with marginal $q$.

This data defines a family of models that clearly satisfy (i), (ii) and (iii). In particular, all the maps $f_{\lam,j}^{(\vec{n})}$, $j=1,\ldots,k_{\vec{n}}$, have the same contraction ratio $\prod_{i=1}^k \lambda^{n_i \beta_i}$.

Using the self-similarity relation \eqref{eq:dssr} and the definition of $\PP$, we get
\begin{align*}
\int_\Omega \eta_\lam^\pom\,d\PP(\om) &= \int_\Omega \sum_{j=1}^{k_{\om_1}} \wtil{p}_j^{(\om_1)} f_{\lam,j}^{(\om_1)}(\eta_\lam^{(T\om)})\,d\PP(\om)\,\\
&= \sum_{\vec{n}\in I} q_{\vec{n}} \int_\Omega \sum_{j=1}^{k_{\vec{n}}} \wtil{p}_j^{(\vec{n})} f_{\lam,j}^{(\vec{n})}(\eta_\lam^{(\om)})\,d\PP(\om)\\
&= \sum_{\vec{n}\in I} \sum_{j=1}^{k_{\vec{n}}} p_{u_j^{(\vec{n})}} \int_\Omega g_{\lam,u_j^{(\vec{n})}}(\eta_\lam^\pom)\,d\PP(\om)\\
&= \sum_{u\in\{1,\ldots,k\}^r}  p_u \int_\Omega \,g_{\lam,u}(\eta_\lam^\pom)\,d\PP(\om)\\
&= \sum_{u\in\{1,\ldots,k\}^r}  p_u\, g_{\lam,u}\left(\int_\Omega \eta_\lam^\pom\,d\PP(\om)\right),
\end{align*}
which, by uniqueness of self-similar measures, establishes (iv).

It remains to estimate the similarity dimension of $\Sigma_\lam$. For $u\in\{1,\ldots,k\}^r$, we write $\lam_u = \lam^{\beta_{u_1}+\ldots+\beta_{u_r}}$, i.e. $\lam_u$ is the contraction ratio of $g_{\lam,u}$. With this notation in hand, we have
\begin{align*}
\textup{s-dim}(\Sigma_\lam) &= \frac{-\sum_{\vec{n}\in I} q_{\vec{n}} H(\wtil{p}_{\vec{n}}) }{\sum_{\vec{n}\in I} q_{\vec{n}} \log\prod_{j=1}^k \lambda^{ n_j \beta_j}}\\
&=\frac{\sum_{\vec{n}\in I} \sum_{i=1}^{k_{\vec{n}}}   p_{u_i^{(\vec{n})}} \log( p_{u_i^{(\vec{n})}}/q_{\vec{n}})}{\sum_{u\in\{1,\ldots,k\}^r}p_u\log(\lam_u) }\\
&=\frac{- H\left((p_u)_{u\in\{1,\ldots,k\}^r}\right) -  \sum_{\vec{n}\in I}q_{\vec{n}}\log(q_{\vec{n}})}{\sum_{u\in\{1,\ldots,k\}^r}p_u\log(\lam_u) }\\
&\ge \frac{- H\left((p_u)_{u\in\{1,\ldots,k\}^r}\right) +\log|I|}{\sum_{u\in\{1,\ldots,k\}^r}p_u\log(\lam_u) }\\
&\ge \frac{-r H(\mathbf{p})+\log((r+1)^k)H(\mathbf{p})/H(\mathbf{p})}{r \sum_{i=1}^k p_i \log(\lam^{-\beta_i})}\\
&=  \frac{r- \log((r+1)^k)/H(\mathbf{p})}{r} \, s(\lam,\mathbf{p}),
\end{align*}
which shows that (v) holds provided $r$ is taken large enough in terms of $\delta$ and $H(\mathbf{p})$.
\end{proof}

We go back to the proof of Proposition \ref{prop:non-hom}. Let $\delta=\eps/2$ and choose $r \in \N$ so that Lemma \ref{lem:non-hom-to-random-hom} holds for this choice of $\delta$ and $\mathbf{p}$. Let $(\Sigma_\lam)$ be the models given by Lemma \ref{lem:non-hom-to-random-hom}. Using parts (i),(ii) of the lemma, write
\begin{equation} \label{eq:functions-random-non-hom}
 \Phi_\lam^{(i)} =  (\lam^{\gamma_i} x+\wtil{t}_{\lam,j}^{(i)})_{1\le j\le k_i},
\end{equation}
where $\gamma_i>0$ depend on $\beta_1,\ldots,\beta_k$. Note that, once $\lam$ and $\om$ have been fixed, the measure $\eta_\lam^\pom$ can be interpreted as the distribution of the random sum
\[
\sum_{n\in\N}  \left(\prod_{j=1}^{n-1}\lambda^{\gamma_{\om_j}}\right)\wtil{t}^{(\om_n)}_{\lam, u_n},
\]
where the sequence $(u_n)$ is chosen according to the product measure $\oeta^\pom$ (which does not depend on $\lam$). Fix $s\in\N$. By splitting $\sum_{n\in\N}$ into $\sum_{n \in \N,s|n}$ and $\sum_{n \in \N,s\nmid n}$, we can decompose $\eta_\lam^\pom$ as a convolution
\begin{equation} \label{eq:convolution}
\eta_\lam^\pom = (\eta'_\lam)^{\pom}*(\eta''_\lam)^{\pom}.
\end{equation}
That is, $(\eta'_\lam)^{\pom}$ corresponds to ``keep only every $s$-th term'' and $(\eta''_\lam)^{\pom}$ to ``skip every $s$-th term'' in the construction of $\eta_\lam^\pom$. Formally, define two maps $\wtil{\Pi}'_{\lam,\om}, \wtil{\Pi}''_{\lam,\om}:\Omega\to\R$ as follows:
\begin{align*}
\wtil{\Pi}'_{\lam,\om}(u) &= \sum_{n \in \N,s|n} \left(\prod_{j=1}^{n-1}\lambda^{\gamma_{\om_j}}\right)\wtil{t}^{(\om_n)}_{\lam, u_n}, \\
\wtil{\Pi}''_{\lam,\om}(u)&= \sum_{n \in \N,s\nmid n} \left(\prod_{j=1}^{n-1}\lambda^{\gamma_{\om_j}}\right)\wtil{t}^{(\om_n)}_{\lam, u_n}.
\end{align*}
Then \eqref{eq:convolution} holds for $(\eta'_\lam)^{\pom}=\wtil{\Pi}'_{\lam,\om}(\oeta^\pom)$ and $(\eta''_\lam)^{\pom}=\wtil{\Pi}''_{\lam,\om}(\oeta^\pom)$.

Recalling the decomposition given by Lemma \ref{lem:non-hom-to-random-hom}(iv), notice that if $\eta_\lam^\pom$ is absolutely continuous for $\PP$-almost all $\om$, then so is $\nu_\lam^{\mathbf{p}}$ (since if $A$ is Lebesgue-null, then $\eta_\lam^\pom(A)=0$ for $\PP$-almost all $\om$, and so $\nu_\lam^{\mathbf{p}}(A)=0$). To establish absolute continuity of $\eta_\lam^\pom$, we will rely on the following result from \cite{Shmerkin14}.
\begin{lemma}
 Let $\mu',\mu''$ be two Borel probability measures on $\R$ such that $\mu'\in\mathcal{D}_1$ and $\dim_H\mu''=1$. Then $\mu'*\mu''$ is absolutely continuous.
\end{lemma}
In light of this lemma  and \eqref{eq:convolution}, the proof of Proposition \ref{prop:non-hom} (and therefore of Theorem \ref{thm:abs-cont-non-hom}) will be finished once we have proved the following two lemmas:

\begin{lemma} \label{lem:as-Fourier-decay}
 For any choice of $s\in\N$, for almost all $\lam\in (0,1)$, the measure $(\eta'_\lam)^\pom$ is in $\mathcal{D}_1$ for $\PP$-almost all $\om$.
 \end{lemma}

\begin{lemma} \label{lem:as-full-dim}
 There is $s\in \N$ such that for almost all $\lam\in (0,1)$ such that $s(\lam,\mathbf{p})>1+\eps$, the measure $(\eta''_\lam)^\pom$ has exact dimension $1$ for $\PP$-almost all $\om$.
\end{lemma}

\begin{proof}[Proof of Lemma \ref{lem:as-Fourier-decay}]
The idea is to realize $(\eta'_\lam)^{\pom}$ as a family of random measures generated by a related model $\Sigma'_\lam$, and apply Theorem \ref{thm:Fourier-decay} to $\Sigma'_\lam$, together with Fubini.

We begin by defining $\Sigma'_\lam$. Let $I'=I^s$ and, for $\vec{i}\in I'$, consider the IFS
 \[
(\Phi'_\lam)^{(\vec{i})} =  \left(\lam_{\vec{i}_1}\cdots \lam_{\vec{i}_s} x + \wtil{t}_{\lam,j}^{(\vec{i}_s)}\right)_{j=1}^{k_{\vec{i}_s}},
\]
Further, let $\widetilde{p}'_{\vec{i}}=\widetilde{p}_{\vec{i}_s}$, and set
\[
F(\om) = \left( (\om_{js+1},\ldots, \om_{(j+1)s}) \right)_{j=0}^\infty \in (I')^\N.
\]
It is a consequence of the definitions that if we denote the random measures generated by the model $\Sigma'_\lam=\left(((\Phi'_\lam)^{(\vec{i})})_{\vec{i} \in I'}, (\widetilde{p}'_{\vec{i}})_{\vec{i} \in I'}, F(\PP)\right)$  by $\wtil{\eta}_\lam^{\om'}$, then
\begin{equation} \label{eq:prime-model-to-old-model}
 \wtil{\eta}_\lam^{(F(\om))}=(\eta'_\lam)^{\pom}
\end{equation}

Note that the contraction ratio of the maps in $(\Phi'_\lam)^{(\vec{i})}$ is $\lam^{\beta'_{\vec{i}}}$, where $\beta'_{\vec{i}}>0$ are numbers that depend only on $s$ and the $\gamma_i$, which in turn depended only on $r$ and the $\beta_i$ (recall Lemma \ref{lem:non-hom-to-random-hom}). It also follows from Lemma \ref{lem:non-hom-to-random-hom}(ii)  and the definition of the $f_{\lambda,j}^{(\vec{i})}$ that, for any choice of $s$ and for any $\lam\in (0,1)$, the model $\Sigma'_\lam$ is non-degenerate. Finally, we note that $F(\PP)$ is a (globally supported) Bernoulli measure.

Let $\mathcal{L}$ be Lebesgue measure on $(0,1)$. We can then apply Theorem \ref{thm:Fourier-decay}, the identity \eqref{eq:prime-model-to-old-model} and Fubini's Theorem to conclude that there exists a full $(\PP\times\mathcal{L})$-measure Borel set $\mathcal{G}\subset \Omega\times(0,1)$, such that $(\eta'_\lam)^{\pom}\in\mathcal{D}_1$ for $(\om,\lam)\in\mathcal{G}$. We remark that, even though the translations in the model $\Sigma'_\lam$ also depend on the parameter $\lam$, in Theorem \ref{thm:Fourier-decay} the exceptional set is independent of the translations (as long as the model stays non-degenerate). We can apply again Fubini's Theorem to reach the desired conclusion.
\end{proof}

The proof of Lemma \ref{lem:as-full-dim} relies on a result of M. Hochman \cite[Theorem 1.8]{Hochman14}. We state only the consequence we will require.
\begin{lemma} \label{lem:dim-0-scc}
 If $h_i(x)=\lam_i x+t_i$ are affine maps, write $\wtil{\Delta}(h_1,h_2)=1$ if $\lam_1\neq \lam_2$ and $\wtil{\Delta}(h_1,h_2)=|t_1-t_2|$ if $\lam_1=\lam_2$. Given $\lam\in (0,1)$, define
\[
\wtil{\Delta}_n(\lam) = \min_{u\neq u'\in\{1,\ldots,k\}^n} \wtil{\Delta}(g_{\lam,u},g_{\lam,u'}).
\]
Then the following holds: the set
\[
E'=E'(\beta_1,\ldots,\beta_k) = \{ \lam\in (0,1): \log\wtil{\Delta}_n(\lam)/n\rightarrow -\infty\}
\]
has zero Hausdorff dimension (in particular, zero Lebesgue measure).
\end{lemma}
\begin{proof}
Note that for any distinct $u,u'\in\{1,\ldots,k\}^\N$, the map
\[
\psi_{u,u'}(\lam) = \sum_{n=0}^\infty \lam^{\beta_{u_1}}\cdots \lam^{\beta_{u_n}} t_{u_{n+1}} -\lam^{\beta_{u'_1}}\cdots \lam^{\beta_{u'_n}} t_{u'_{n+1}}
\]
is non-zero. Indeed, using that the translations $t_i$ are all different, and looking at the first place where the sequences $u$ and $u'$ differ, we can write $\psi_{u,u'}$ in the form $\sum_j  v_j\lam^{\zeta_j}$ for some non-zero numbers $v_j$ and a strictly decreasing (finite or infinite, but nonempty) sequence of $\zeta_j \ge 0$. This shows that the assumption of \cite[Theorem 1.8]{Hochman14} is met, giving the claim.
\end{proof}

\begin{proof}[Proof of Lemma \ref{lem:as-full-dim}]
 The idea is similar to the proof of Lemma \ref{lem:as-Fourier-decay}, but relying on Theorem \ref{thm:sup-exp-conc} together with Lemma \ref{lem:dim-0-scc} instead.

For the time being we let $s\in\N$ be arbitrary. Write $I''=I^s$, for each $\vec{i}\in I''$ let $K_{\vec{i}}= \prod_{\ell=1}^{s-1} \{1,\ldots,k_{\vec{i}_\ell}\}$, and consider the IFS
 \[
  (\Phi'')^{(\vec{i})} = \left(\lam_{\vec{i}_1}\cdots \lam_{\vec{i}_s} x  + \sum_{\ell=1}^{s-1} \lam_{\vec{i}_1}\cdots \lam_{\vec{i}_{\ell-1}}  \wtil{t}_{\lam,j_\ell}^{(\vec{i}_\ell)}:j\in K_{\vec{i}}\right).
 \]
Let also $\wtil{p}''_{\vec{i}}$ be the uniform probability vector on $K_{\vec{i}}$
 and (once again)
 \[
  F(\om) = \left( (\om_{js+1},\ldots, \om_{(j+1)s}) \right)_{j=0}^\infty \in (I'')^\N.
 \]
 Abusing notation slightly, for any $\ell\in\N$ we will also define $F:I^{s \ell}\to (I'')^\ell$ in the obvious way:
 \[
  F(\om) =\left( (\om_{js+1},\ldots, \om_{(j+1)s}) \right)_{j=0}^{\ell-1}.
 \]

 Similar to Lemma \ref{lem:as-Fourier-decay}, an inspection of the definitions shows that, denoting the measures generated by the model $\Sigma''_\lam=\left(((\Phi''_\lam)^{(\vec{i})})_{\vec{i} \in I''}, (\widetilde{p}''_{\vec{i}})_{\vec{i} \in I''}, F(\PP)\right)$  by $\wtil{\eta}_\lam^{\om''}$ (we use the same notation as for the measures in Lemma \ref{lem:as-Fourier-decay}, but the measures are different), it holds that
\begin{equation} \label{eq:double-prime-model-to-old-model}
 \wtil{\eta}_\lam^{(F(\om))}=(\eta''_\lam)^{\pom}.
\end{equation}
Observe that for any $\om\in I^{s\ell}$,
\begin{equation}\label{eq:ineqdel}
 \Delta_\ell^{(F(\om))}(\Sigma''_\lam) \ge \Delta_{s\ell}^\pom(\Sigma_\lam).
\end{equation}
To see this, let $(\X''_{\lam,\ell})^\pom$ denote the finite code spaces for the model $X''_\lam$, and note that if $u:=(u^{(j)})_{j=1}^\ell, v= (v^{(j)})_{j=1}^\ell \in (\X''_{\lam,\ell})^\pom$ and we write
\begin{align*}
\wtil{u} &= \left(u^{(1)}_1,\ldots, u^{(1)}_{s-1},i_0,\ldots,u^{(\ell)}_1,\ldots, u^{(\ell)}_{s-1},i_0\right),\\
\wtil{v} &= \left(v^{(1)}_1,\ldots, v^{(1)}_{s-1},i_0,\ldots,v^{(\ell)}_1,\ldots, v^{(\ell)}_{s-1},i_0\right)
\end{align*}
where $i_0$ is an arbitrary element of $I$, then
\[
(t''_u)^{(F(\om))} - (t''_v)^{(F(\om))} = t_{\wtil{u}}^\pom-t_{\wtil{v}}^\pom,
\]
where $t, t''$ denote the translations with respect to models $\Sigma_\lam,\Sigma''_\lam$, respectively. Indeed, this holds simply because the model $\Sigma''_\lam$ is obtained by ``skipping every $s$-th term and recoding'' from $\Sigma_\lam$. Recalling \eqref{eq:def-exp-concentration}, we see that \eqref{eq:ineqdel} holds.

In turn, it follows from Lemma \ref{lem:non-hom-to-random-hom}(ii) that for any $\om\in I^{s\ell}$,
\[
  \Delta_{s\ell}^\pom(\Sigma_\lam)  \ge \wtil{\Delta}_{rs\ell}(\lam).
\]
By Lemma \ref{lem:dim-0-scc}, and since $\wtil{\Delta}_{n}$ is non-increasing in $n$, there is a null set $E\subset (0,1)$ such that $\log \wtil{\Delta}_{rs\ell}(\lam)/\ell \nrightarrow -\infty$ for all $\lam\in (0,1)\setminus E$. We deduce from the above discussion that, given $\lam\in (0,1)\setminus E$, there are a sequence $n_j\to\infty$ and $M>-\infty$ such that
\[
\log \Delta_{n_j}^{(\om'')}(\Sigma''_\lam)/{n_j} \ge M \quad\text{for all }j \text{ and all }\om''\in (I'')^{n_j} \text{ such that } |(\X_{\lam, n_j}'')^{(\om'')}|>1.
\]
Since for $F(\PP)$-almost all $\om''$ it holds that $|(\X_{\lam, n}'')^{(\om'')}|>1$ for $n\ge n(\om'')$, we have verified that the hypotheses of  Theorem \ref{thm:sup-exp-conc} hold for $\Sigma''_\lam$. We deduce from Theorems \ref{thm:exact} and \ref{thm:sup-exp-conc} and the identity \eqref{eq:double-prime-model-to-old-model} that for all $\lam\in(0,1)\setminus E$ and $\PP$-almost all $\om$, the measure $(\eta''_\lam)^{\pom}$ has exact dimension $\min(\textup{s-dim}(\Sigma''_\lam),1)$.

Hence, taking into account Lemma \ref{lem:non-hom-to-random-hom}(v) and our choice $\delta=\eps/2$, in order to finish the proof it remains to show that if $s$ is chosen large enough, then $\textup{s-dim}(\Sigma''_\lam)>1$ provided that $\textup{s-dim}(\Sigma_\lam)>1+\eps/3$. Firstly, we note that for $\vec{i}\in I''=I^s$,
\[
H(\wtil{p}''_{\vec{i}})  = \log |K_{\vec{i}}| = \sum_{\ell=1}^{s-1} \log k_{i_\ell}
= \sum_{\ell=1}^{s-1} H(\wtil{p}_{\vec{i}_\ell}).
\]
Hence, denoting the marginal of $\PP$ by $(q_i)_{i\in I}$ and recalling \eqref{eq:functions-random-non-hom}, we can calculate:
\begin{align*}
\textup{s-dim}(\Sigma''_\lam) &= \frac{\sum_{\vec{i}\in I^s} H(\wtil{p}''_{\vec{i}})\, q_{\vec{i}_1}\cdots q_{\vec{i}_s}    }{-\sum_{\vec{i}\in I^s} \log(\lam^{\gamma_{\vec{i}_1}}\cdots \lam^{\gamma_{\vec{i}_s}})\, q_{\vec{i}_1}\cdots q_{\vec{i}_s} }\\
&= \frac{(s-1)\sum_{i\in I} H(\wtil{p}_i) q_i    }{-s \sum_{i\in I} \log(\lam^{\gamma_i}) q_i}\\
&= (1-1/s) \,\textup{s-dim}(\Sigma_\lam).
\end{align*}
This confirms that we can ensure that $\textup{s-dim}(\Sigma''_\lam)>1$ by taking $s$ large enough in terms of $\eps$, completing the proof.
\end{proof}


\begin{thebibliography}{10}

\bibitem{BufetovSolomyak15}
Alexander~I. Bufetov and Boris Solomyak.
\newblock The {H}\"{o}lder property for the spectrum of translation flows in genus
  two.
\newblock {\em Israel J. Math}, 223(1):205--259, 2018.


\bibitem{EvansGariepy15}
Lawrence~C. Evans and Ronald~F. Gariepy.
\newblock {\em Measure theory and fine properties of functions}.
\newblock Textbooks in Mathematics. CRC Press, Boca Raton, FL, revised edition,
  2015.

\bibitem{FalconerJin14}
Kenneth~J. Falconer and Xiong Jin.
\newblock Exact dimensionality and projections of random self-similar measures
  and sets.
\newblock {\em J. Lond. Math. Soc. (2)}, 90(2):388--412, 2014.

\bibitem{FengHu09}
De-Jun Feng and Huyi Hu.
\newblock Dimension theory of iterated function systems.
\newblock {\em Comm. Pure Appl. Math.}, 62(11):1435--1500, 2009.

\bibitem{furstenberg2008ergodic}
Hillel Furstenberg.
\newblock Ergodic fractal measures and dimension conservation.
\newblock {\em Ergodic Theory and Dynamical Systems}, 28(02):405--422, 2008.

\bibitem{GSSY16}
Daniel Galicer, Santiago Saglietti, Pablo Shmerkin, and Alexia Yavicoli.
\newblock {$L^q$} dimensions and projections of random measures.
\newblock {\em Nonlinearity}, 29(9):2609--2640, 2016.

\bibitem{Hochman14}
Michael Hochman.
\newblock On self-similar sets with overlaps and inverse theorems for entropy.
\newblock {\em Ann. of Math. (2)}, 180(2):773--822, 2014.

\bibitem{Hochman15}
Michael Hochman.
\newblock On self-similar sets with overlaps and inverse theorems for entropy
  in $\mathbb{R}^d$.
\newblock {\em Mem. Amer. Math. Soc.}, To appear, 2015.
\newblock arXiv:1503.09043.

\bibitem{Maker40}
Philip~T. Maker.
\newblock The ergodic theorem for a sequence of functions.
\newblock {\em Duke Math. J.}, 6:27--30, 1940.

\bibitem{Neun_2001}
J\"{o}rg~Neunh{\"a}userer.
\newblock Properties of some overlapping self-similar and some self-affine
  measures.
\newblock {\em Acta Math. Hungar.}, 92(1-2):143--161, 2001.

\bibitem{NgaiWang05}
Sze-Man Ngai and Yang Wang.
\newblock Self-similar measures associated to {IFS} with non-uniform
  contraction ratios.
\newblock {\em Asian J. Math.}, 9(2):227--244, 2005.

\bibitem{PeresSolomyak96}
Yuval Peres and Boris Solomyak.
\newblock Absolute continuity of {B}ernoulli convolutions, a simple proof.
\newblock {\em Math. Res. Lett.}, 3(2):231--239, 1996.

\bibitem{Shmerkin14}
Pablo Shmerkin.
\newblock On the exceptional set for absolute continuity of {B}ernoulli
  convolutions.
\newblock {\em Geom. Funct. Anal.}, 24(3):946--958, 2014.

\bibitem{Shmerkin16}
Pablo Shmerkin.
\newblock On {F}urstenberg's intersection conjecture, self-similar measures,
  and the ${L}^q$ norms of convolutions.
\newblock {\em Ann. of Math. (2)}, accepted for publication.

\bibitem{ShmerkinSolomyak16}
Pablo Shmerkin and Boris Solomyak.
\newblock Absolute continuity of self-similar measures, their projections and
  convolutions.
\newblock {\em Trans. Amer. Math. Soc.}, 368(7):5125--5151, 2016.

\bibitem{Simmons12}
David Simmons.
\newblock Conditional measures and conditional expectation; {R}ohlin's
  disintegration theorem.
\newblock {\em Discrete Contin. Dyn. Syst.}, 32(7):2565--2582, 2012.

\bibitem{SimonVago17}
K\'{a}roly Simon and Lajos V\'{a}g\'{o}.
\newblock Singularity versus exact overlaps for self-similar measures.
\newblock Preprint, arXiv:1702.06785, 2017.

\bibitem{Solomyak95}
Boris Solomyak.
\newblock On the random series {$\sum\pm\lambda\sp n$} (an {E}rd{\H o}s
  problem).
\newblock {\em Ann. of Math. (2)}, 142(3):611--625, 1995.

\bibitem{Tho}
Hermann Thorisson
\newblock {\em Coupling, stationarity, and regeneration.}
\newblock {Springer-Verlag, New York, 2000.}

\bibitem{Varju16}
P\'{e}ter Varj\'{u}.
\newblock Absolute continuity of {B}ernoulli convolutions for algebraic
  parameters.
\newblock Preprint, arXiv:1602.00261.

\bibitem{Walters82}
Peter Walters.
\newblock {\em An introduction to ergodic theory}, volume~79 of {\em Graduate
  Texts in Mathematics}.
\newblock Springer-Verlag, New York-Berlin, 1982.

\end{thebibliography}

\end{document}